\documentclass[12pt,reqno]{amsart}

\usepackage{amsmath,amssymb,amsfonts,amsthm}

\usepackage{eurosym}
\usepackage{amsfonts}
\usepackage{amsmath,amssymb,amsfonts,mathtools,paralist}
\usepackage{amsfonts}
\usepackage{graphicx}
\usepackage{graphics}
\usepackage{url,hyperref}
\usepackage{xcolor}

\newtheorem{theorem}{Theorem}

\newtheorem{corollary}[theorem]{Corollary}

\newtheorem{definition}[theorem]{Definition}
\newtheorem{example}[theorem]{Example}

\newtheorem{notation}[theorem]{Notation}

\newtheorem{proposition}[theorem]{Proposition}
\newtheorem{remark}[theorem]{Remark}

\begin{document}

\title[Oscillator]{Sub-Riemannian and sub-Lorentzian geodesics \\
of the oscillator groups}
\thanks{This work was supported by Consejo Nacional de Investigaciones
Cient\'ificas y T\'ecnicas and Secre\-tar\'{\i}a de Ciencia y T\'ecnica de la
Universidad Nacional de C\'ordoba. The collaboration would have not been 
possible without the fellowship awarded to the second author by the IMU-Simons Research 
 Program for Developing Countries}

\author{Mauricio Godoy Molina and Marcos Salvai}
\date{}

\begin{abstract}
The oscillator groups are four-dimensional solvable Lie groups, extensions
of the Heisenberg Lie group. We present them as circle bundles of the
standard contact sub-Riemannian structure on $\mathbb{R}^{3}$. We define
sub-Riemannian and sub-Lorentzian structures on them, describe their
geodesics explicitly and determine which of them are periodic.

We also study a generalization: 
Let $G$ be a semisimple Lie group and $K$ a compact subgroup such that 
$\left( G,K\right) $ is a symmetric
pair of the compact or the noncompact type. Let $\mathfrak{g}$ and $%
\mathfrak{k}$ be the Lie algebras of $G$ and $K$,  respectively.
We consider on $\mathfrak{g}$
the usual nilpotent Lie group structure  with center $\mathfrak{k}$ 
and call it $N\left( \mathfrak{g},\mathfrak{k}\right)$.  
A suitable semidirect product with $\operatorname{Ad}\left( K\right) $
yields a solvable Lie group $\operatorname{Osc}\left( G,K\right) $, which
generalizes the oscillator groups and is also quadratic (that is, it
possesses bi-invariant metrics; in particular, a canonical one). We define
nonholonomic pseudo-Riemannian structures on it and find their geodesics explicitly.
In doing so, we obtain a result that may be not merely auxiliary: formulas
for the monoparametric subgroups of $\operatorname{Osc}\left( G,K\right) $ and for
the sub-Riemannian geodesics of $N\left( \mathfrak{g},\mathfrak{k}\right) $ with the standard
left-invariant distribution $\mathcal{D}$. Moreover, when $G/K$ is compact with
rank one, we realize $\operatorname{Osc}\left( G,K\right) $ as a Stiefel bundle of 
$\left( N\left( \mathfrak{g},\mathfrak{k}\right) ,\mathcal{D}\right) $, up to finite coverings.
\end{abstract}


\maketitle

%
%

\section{Introduction}

Sub-Riemannian geometry is a branch of differential geometry close to
control theory. It deals with systems in which objects are allowed to move only
in certain directions, but can still reach any point in space. The
restriction in the displacements is given by a linear condition and the cost of
admissible motions is defined. %
(it is quadratic at the infinitesimal level).
It is usually an interesting problem to find the motions that locally
minimize the cost. For generalities on sub-Riemannian geometry see \cite%
{RM} and \cite{AgrachevLibro}. In the case of no restrictions, the
theory reduces to Riemannian geometry. The Lorentzian and,
more generally, pseudo-Riemannian counterparts have also been considered
worthy of study. In this case, there may be no cost to
minimize, but there are still distinguished curves, induced by a Hamiltonian
defined analogously as in the positive definite case.

The oscillator groups are four-dimensional nonnilpotent solvable Lie groups,
semi\-direct products of the Heisenberg group with a one-parameter group. They
were introduced in \cite{Streater} in 1967. They have the remarkable
property of possessing bi-invariant Lorentz metrics. As far a we know,
sub-Riemannian or sub-Lorentzian structures on these groups have been studied
only in \cite{BiggsNagy}, with other approach and different results.  
Concerning
these structures on general solvable Lie groups which are not nilpotent,
there are thorough studies in dimension three; see \cite{SachkovSE2,SachkovSol,LokuNnS,LinoGrama}. 
But aside from those papers, the literature on
the subject seems to be quite limited.

We present the oscillator groups $\operatorname{Osc}\left( \kappa \right) $ ($\kappa
=\pm 1$) as circle bundles of the canonical contact distributions on $%
\mathbb{R}^{3}$ with the natural metric (Proposition \ref{identificacion}).
This allows for visualization and may be useful to prompt new questions. We
define a family of sub-Riemannian or sub-Lorentzian structures on them and
compute their geodesics explicitly (Theorem \ref{subOsc}), providing new
applications of a well-established technique (valid under
certain conditions, with an ambient space with a bi-invariant
pseudo-Riemannian metric). We also determine which geodesics are periodic
(Proposition \ref{PropoPeriodic}).

In Section \ref{SectionHD} we study higher dimensional generalizations. Let $%
\left( G,K\right) $ be a semisimple symmetric pair, either of the compact or
the noncompact type, of any rank, with $K$ compact. We define the associated
oscillator group $\operatorname{Osc}\left( G,K\right) $ as an extension 
by $\operatorname{Ad}%
_{G}\left( K\right) $ of the nilpotent Lie group $N\left( \mathfrak{g},%
\mathfrak{k}\right) $, which is the Lie algebra $\mathfrak{g}$ of $G$
endowed with a natural multiplication with center $\mathfrak{k}$. This is
a particular case of the important method called double extension,
introduced in \cite{Medina, Favre}, which led to the
complete description of Lie algebras endowed with ad-invariant metrics. 
The groups  $\operatorname{Osc}\left( G,K\right) $ are very particular cases 
of the interesting class of symmetric spaces whose transvection groups are not semisimple; 
see \cite{Kath} and references therein.

Following the approach of presenting $\operatorname{Osc}\left( \kappa \right) $ as
circle bundles, for a rank-one classical pair $\left( G,K\right) $, we
identify $\operatorname{Osc}\left( K,G\right) $ with some Stiefel bundle on $N\left( 
\mathfrak{g},\mathfrak{k}\right) $ (Proposition \ref{PropStiefel}). Finally,
we define nonholonomic pseudo-Riemannian structures on $\operatorname{Osc}\left(
G,K\right) $ (of which the ones considered on $\operatorname{Osc}\left( \kappa
\right) $ are particular cases) and describe their geodesics (Theorem \ref%
{sub-semi}).

Auxiliary for this paper, but of interest in itself, is the explicit
computation of the sub-Riemannian geodesics of the Carnot group $N\left( 
\mathfrak{g},\mathfrak{k}\right) $ (Theorem \ref{geoSubN}). Moreover, we
realize them as projections of orbits of the exponential mapping of the
extended group $\operatorname{Osc}\left( G,K\right) $ (Theorem \ref{ThmGamma} and
Corollay \ref{Coro}). In the
particular case $\mathfrak{g}$ $=\mathfrak{o}\left( 1+m\right) $, $\mathfrak{k}=\mathfrak{o}\left(
m\right) $, the group $N\left( \mathfrak{g},\mathfrak{k}\right) $ is the
free, step-two Carnot group of rank $m$ and its sub-Riemannian geometry is
well understood, for instance in \cite{RizziCut}.  In Section \ref{HighNilpotent} below 
we give other references, containing formulas for the
geodesics, for $m$ small. Perhaps our general
approach provides a more conceptual perspective on the topic.

Regarding notation, we denote distributions on nilpotent groups with the
letter $\mathcal{D}$ and on the solvable extensions with the letter $%
\mathcal{E}$.

We would like to thank Patrick Massot for a picture of his that inspired 
Fig.\ \ref{fig1} and Ramiro Salvai for drawing the latter based on the former.

We are deeply grateful to the International Mathematical Union and the
Simons Foundation. Through the Research Fellowship Program for Developing
Countries they made the second author's stay at the Universidad de la
Frontera possible. He expresses his gratitude for the warm hospitality he
experienced there.

\section{The oscillator groups}

We fix $\kappa \in \left\{ 1,-1\right\} $ and consider on $\mathbb{R}^{3}$
the distribution $\mathcal{D}^{\kappa }=$ span$\left\{ \mathcal{X},\mathcal{Y%
}\right\} $, where the vector fields $\mathcal{X},\mathcal{Y}$ are%
\begin{equation}
\mathcal{X}\left( x,y,v\right) =\left( 1,0,-\tfrac{1}{2}\kappa y\right) 
\text{\ \ \ \ and\ \ \ }\mathcal{Y}\left( x,y,v\right) =\left( 0,1,\tfrac{1}{%
2}\kappa x\right)  \label{xieta}
\end{equation}%
(for the sake of simplicity, sometimes the dependence on $\kappa $ remains
tacit). This is one of the presentations of the canonical contact form on $%
\mathbb{R}^{3}$. As customary, for each $p\in \mathbb{R}^{3}$ we endow $%
\mathcal{D}_{p}^{\kappa }$ with the inner product and the orientation such that $\left\{ 
\mathcal{X}_{p},\mathcal{Y}_{p}\right\} $ is a positively oriented orthonormal basis
(equivalently, such that the canonical projection of $\mathcal{D}%
_{p}^{\kappa }$ onto $\mathbb{R}^{2}$ is a linear isometry 
preserving the orientation). Let $\mathcal{D}%
^{\kappa ,1}$ be the unit circle bundle%
\begin{equation}
\mathcal{D}^{\kappa ,1}=\left\{ \left( p,W\right) \mid p\in \mathbb{R}^{3}%
\text{, }W\in \mathcal{D}_{p}^{\kappa }\text{ with }\left\Vert W\right\Vert
=1\right\}  \label{UnitCB}
\end{equation}

Figure \ref{fig1} illustrates $\mathcal{D}^{\kappa ,1}$ in the case $%
\kappa =-1$; the fibers are round circles with respect to the inner product on the
distribution (we have compressed the vertical scale to make it look nicer).

\begin{figure}[h!t]
	\centering{
		\includegraphics[totalheight=2.7in,angle=0]
		{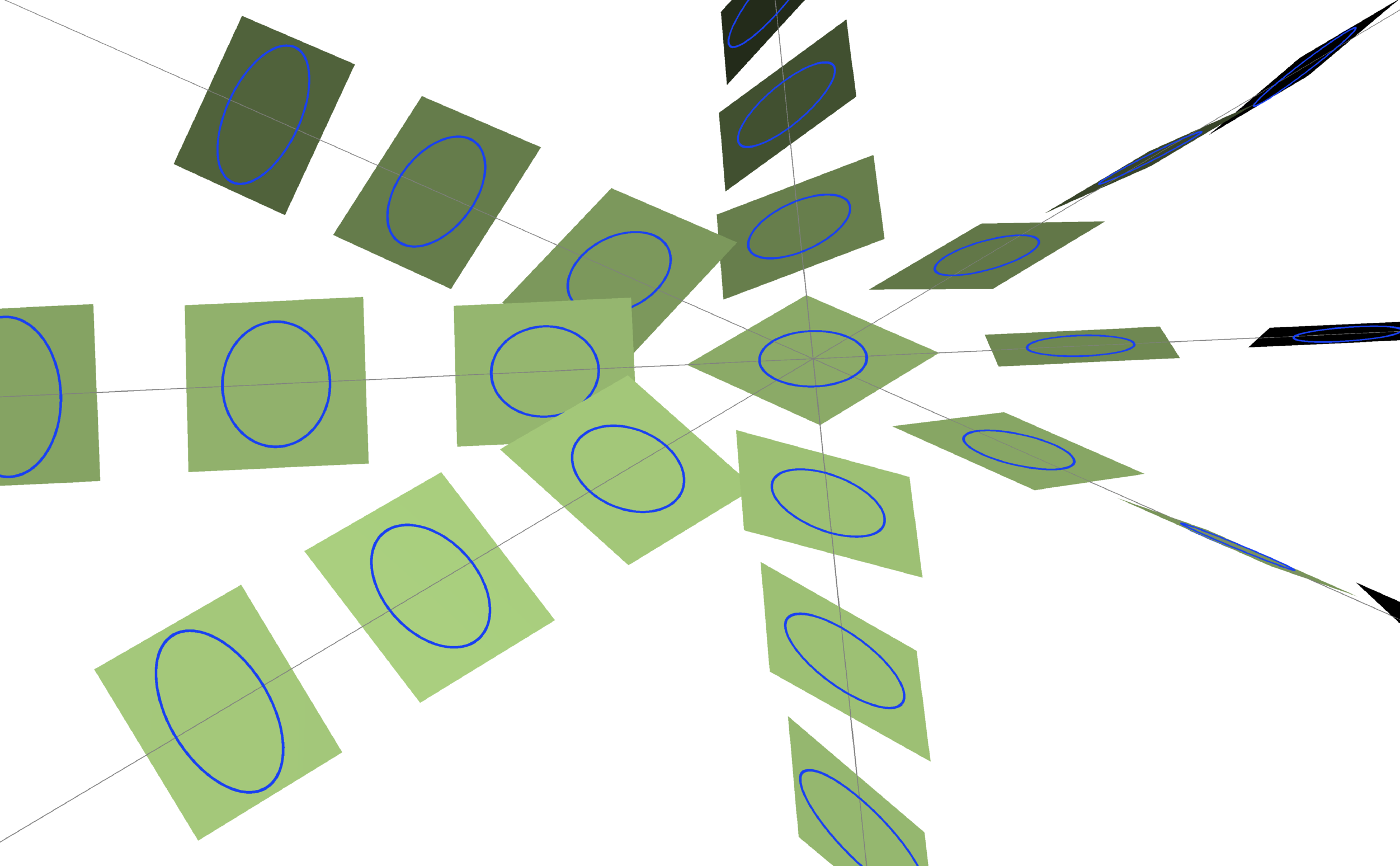}}
	\caption{The oscillator group as a circle bundle. (AI-generated image created with Claude, inspired by a drawing by Patrick Massot.)}
	\label{fig1}
\end{figure}

Let $\Phi _{\kappa }:\mathbb{R}^{4}\rightarrow \mathbb{R}^{4\times 4}$ be
defined by $\Phi _{\kappa }\left( u,x,y,v\right) =\left( 
\begin{array}{cc}
1 & 0 \\ 
a & A%
\end{array}%
\right) $, where $0$ is the row vector $\left( 0,0,0\right) $ and%
\begin{equation}
a=\left( 
\begin{array}{c}
x \\ 
y \\ 
v%
\end{array}%
\right) \text{,\ \ \ \ \ }A=\left( 
\begin{array}{ccc}
\cos u & -\sin u & 0 \\ 
\sin u & \cos u & 0 \\ 
\tfrac{1}{2}\kappa \left( x\sin u-y\cos u\right) & \tfrac{1}{2}\kappa \left(
x\cos u+y\sin u\right) & 1%
\end{array}%
\right) \text{.}  \label{aA}
\end{equation}

\begin{definition}
The image of $\Phi _{\kappa }$ is a Lie subgroup of $GL\left( 4,\mathbb{R}%
\right) $ called the $\kappa $-oscillator group, denoted by $\operatorname{Osc}%
\left( \kappa \right) .$
\end{definition}

\begin{proposition}
\label{identificacion}The group $\operatorname{Osc}\left( \kappa \right) $ may be
identified with $\mathcal{D}^{\kappa ,1}$ as the orbit of $\left(
0,e_{1}\right) $ under the following simply transitive action:%
\begin{equation*}
\left( 
\begin{array}{cc}
1 & 0 \\ 
a & A%
\end{array}%
\right) \cdot \left( p,W\right) =\left( a+Ap,AW\right) \text{.}
\end{equation*}
\end{proposition}

It is no difficult to show that the action is transitive. The isotropy
subgroup is trivial since the second column of $A$ in (\ref{aA}) can be
recovered from its first column and $a$.

The Lie algebra $\operatorname{osc}\left( \kappa \right) $ of $\operatorname{Osc}\left(
\kappa \right) $ is the image of the mapping $\phi _{\kappa }:=\left( d\Phi
_{\kappa }\right) _{0}:\mathbb{R}^{4}\rightarrow \mathbb{R}^{4\times 4} $
defined analogously as $\Phi _{\kappa }$, but with%
\begin{equation*}
A=\left( 
\begin{array}{ccc}
0 & -u & 0 \\ 
u & 0 & 0 \\ 
-\tfrac{1}{2}\kappa y & \tfrac{1}{2}\kappa x & 0%
\end{array}%
\right) \text{.}
\end{equation*}

Let $\left\{ e_{0},e_{1},e_{2},e_{3}\right\} $ be the canonical basis of $%
\mathbb{R}^{4}$. We call 
\begin{equation}
U=\phi _{\kappa }\left( e_{0}\right) \text{,\qquad }X=\phi _{\kappa }\left(
e_{1}\right) \text{,\qquad }Y=\phi _{\kappa }\left( e_{2}\right) \text{%
,\qquad }V=\phi _{\kappa }\left( e_{3}\right)  \label{UXYV}
\end{equation}%
and have the usual relations defining $\operatorname{osc}\left( \kappa \right) $:%
\begin{equation}
\left[ X,Y\right] =\kappa V\text{,\qquad }\left[ U,X\right] =Y\text{,\qquad }%
\left[ U,Y\right] =-X\text{.}  \label{bracketUV}
\end{equation}

The oscillator group $\operatorname{Osc}\left( \kappa \right) $ has a canonical
bi-invariant Lorentz metric $g_{\kappa }$, given at the identity by 
\begin{equation}
g_{\kappa }\left( \phi _{\kappa }\left( u,x,y,v\right) ,\phi _{\kappa
}\left( \mu ,\xi ,\eta ,\nu  \right) \right) =x\xi +y\eta +\kappa \left(
u\nu  +v\mu \right)  \label{biOsc}
\end{equation}
In particular, the geodesics of $\operatorname{Osc}\left( \kappa \right) $ through
the identity have the form $t\mapsto \exp \left( tZ\right) $ with $Z\in 
\operatorname{osc}\left( \kappa \right) $.

The groups $\operatorname{Osc}\left( 1\right) $ and $\operatorname{Osc}\left( -1\right) $
are inessentially different. Indeed, the mapping $\mathcal{I}:\operatorname{Osc}%
\left( 1\right) \rightarrow \operatorname{Osc}\left( -1\right) $, $\mathcal{I}\left(
\Phi _{1}\left( u,x,y,v\right) \right) =\Phi _{-1}\left( u,x,y,-v\right) $
is a Lie group isomorphism. We deal with both of them (which does not mean
much extra effort) for two reasons. On the one hand, most computations in
the literature are done for $\kappa =1$, but in the nice picture created by
Patrick Massot (see Fig.\ \ref{fig1}) the sense of the asps corresponds to $\kappa =-1$.
On the other hand, we will see that they are very particular cases of the
quadratic solvable Lie groups we introduce in Section \ref{SectionHD}, 
which are associated with symmetric spaces of the compact or the noncompact
type, and for the two-sphere and the hyperbolic plane they may be identified
with $\operatorname{Osc}\left( 1\right) $ and $\operatorname{Osc}\left( -1\right) $,
respectively (see Examples \ref{Example1} and \ref{Example2}).

\subsection{Geodesics of $\operatorname{Osc}\left( \protect\kappa \right) $ in terms
of sub-Riemannian geodesics of the Heisenberg group\label{SubSec1}}

The vector fields $\mathcal{X}$ and $\mathcal{Y}$ on $\mathbb{R}^{3}$ in (%
\ref{xieta}) are left-invariant vector fields of the Lie group $\mathbb{R}%
^{3}$ with multiplication%
\begin{equation*}
\left( x,y,v\right) \cdot \left( \xi ,\eta ,\nu  \right) =\left( x+\xi
,y+\eta ,v+\nu  +\tfrac{\kappa }{2}\left( x\eta -y\xi \right) \right) 
\text{,}
\end{equation*}%
which is the standard presentation of the Heisenberg group and we call $%
\operatorname{Hei}\left( \kappa \right) $. Its Lie algebra bracket is given by%
\begin{equation*}
\left[ \left( x,y,v\right) ,\left( \xi ,\eta ,\nu  \right) \right]
=\left( 0,0,\kappa \left( x\eta -y\xi \right) \right) \text{.}
\end{equation*}

Now, let $N\left(
\kappa \right) =\Phi _{\kappa }\left( \left\{ 0\right\} \times \mathbb{R}%
^{3}\right) \subset GL\left( 4,\mathbb{R}\right) $, which is another presentation 
of the Heisenberg group.
The map $\operatorname{Hei}%
\left( \kappa \right) \rightarrow N\left( \kappa \right) $, $\left(
x,y,v\right) \mapsto \Phi _{\kappa }\left( 0,x,y,v\right) $, is a Lie group
isomorphism sending $\mathcal{X}$ and $\mathcal{Y}$ to the left-invariant
vector fields $X$ and $Y$ on $\operatorname{Osc}\left( \kappa \right) $ defined in (%
\ref{UXYV}), which restrict to vector fields on $N\left( \kappa \right) $.

Consider on $N\left( \kappa \right) $ the canonical sub-Riemannian
structure, that is, the left-invar\-iant distribution $\overline{\mathcal{D}}%
^{\kappa }$ on $N\left( \kappa \right) $ generated by $X,Y$, with $\left\{
X,Y\right\} $ orthonormal. For the definition of sub-Riemannian and
sub-Lorentzian structures and their geodesics, see the Appendix (Section \ref%
{Appendix}). Meanwhile, we comment that in this article, geodesics refer to
normal geodesics, that is, those distinguished curves which are projections
of the integral curves of the symplectic gradient vector field of the
associated Hamiltonian.

A smooth curve $\sigma $ of a manifold with a distribution $\mathcal{E}$ is
said to be admissible for $\mathcal{E}$ if $\sigma ^{\prime }\left( t\right)
\in \mathcal{E}_{\sigma \left( t\right) }$ for all $t$.

Explicit formulas for the geodesics of $\operatorname{Osc}\left( \kappa \right) $
departing form the identity (that is, one-parameter subgroups) and for the
sub-Riemannian geodesics of $N\left( \kappa \right) $ are well-known. See,
for instance, \cite{Gabriela} (or the previous paper \cite{NorbertDorr},
using other coordinates) and Section 1.4.1 in \cite{LibroLeDonne},
respectively. We recall them now, bringing to light how they are related.

We identify $\mathbb{R}^{2}\equiv \mathbb{C}$ in the usual way, and write $%
\operatorname{Osc}\left( \kappa \right) =\Phi _{\kappa }\left( \mathbb{R\times C}%
\times \mathbb{R}\right) $ and $N\left( \kappa \right)=\Phi _{\kappa }\left( \left\{ 0\right\}
\times \mathbb{C\times R}\right) $

\begin{proposition}
\label{PropGeoHeiO}The sub-Riemannian geodesics of $\left( N\left( \kappa
\right) ,\overline{\mathcal{D}}^{\kappa }\right) $ have the form $\Phi
_{\kappa }\circ \left( 0,\sigma _{\left( z,u\right) }\right) $, where either%
\begin{equation}
\sigma _{\left( z,u\right) }\left( s\right) =\left( \tfrac{1}{u}\left(
1-e^{isu}\right) iz,\tfrac{\kappa }{2u^{2}}\left\vert z\right\vert
^{2}\left( us-\sin \left( us\right) \right) \right)  \label{sigmaPropo1}
\end{equation}%
for some $u\in\mathbb{R}, u\neq 0$ and $z\in \mathbb{C}$, or $\sigma _{\left( z,0\right)
}\left( s\right) =\left( sz,0\right) $ for some $z\in \mathbb{C}$.

\smallskip

The geodesic in $\operatorname{Osc}\left( \kappa \right) $ with initial velocity $%
\phi _{\kappa }\left( u,z,v\right) $ is%
\begin{equation}
\gamma \left( s\right) =\Phi _{\kappa }\left( us,\sigma _{\left( z,u\right)
}\left( s\right) +\left( 0,vs\right) \right) \text{.}  \label{gammaPropo1}
\end{equation}
\end{proposition}

For $u\neq 0$, $\sigma _{\left( z,u\right) }$ projects to $\mathbb{C}$ to a
circle through the origin, centered at $iz/u$, with initial velocity $z$,
constant speed $\left\vert z\right\vert $ and radius $\left\vert z/u\right\vert $.

As we mentioned above, the formulas are well-known. Anyway, 
we verify in Example \ref{ExamplegammaGamma} below that
they are in fact the geodesics, illustrating a much more general
result, Theorem \ref{ThmGamma}.


Section 5.4 of \cite{BiggsNagy} establishes relationships between sub-Riemannian
geodesics of the oscillator group and of $SE(2)$ (the group of rigid transformations 
of the plane) for various cases, different from ours.

\subsection{Sub-Riemannian and sub-Lorentzian structures on $\operatorname{Osc}%
\left( \protect\kappa \right) $ and their geodesics}

Let $\lambda \neq 0$ and let $\mathcal{E}^{o}$ and $\mathcal{E}^{\lambda }$
be the left-invariant distributions on $\operatorname{Osc}\left( \kappa \right) $
given at the identity by%
\begin{equation*}
\mathcal{E}_{I_{4}}^{o}=\left\{ \phi _{\kappa }\left( \mu ,\xi ,0,\nu 
\right) \mid \mu ,\xi ,\nu  \in \mathbb{R}\right\} \text{, \ \ \ \ }%
\mathcal{E}_{I_{4}}^{\lambda }=\left\{ \phi _{\kappa }\left( \lambda
\nu  ,\xi ,\eta ,\nu  \right) \mid \xi ,\eta ,\nu  \in \mathbb{R}%
\right\}
\end{equation*}%
and consider on them the restrictions $g_{o}$ and $g_{\lambda }$ of the
bi-invariant Lorentz metric $g_{\kappa }$ (see (\ref{biOsc})) of $\operatorname{Osc}%
\left( \kappa \right) $, respectively.

Before stating the next theorem we recall that a general contact
sub-Riemannian manifold admits a distinguished adapted connection,
introduced in \cite{FVV}. In \cite{Claudio} one can find a presentation
which is independent of sections of the distribution. In our case, the
Heisenberg group $\operatorname{Hei}\left( \kappa \right) $, the left invariant
vector fields $\mathcal{X},\mathcal{Y}$ in (\ref{xieta}) spanning the
distribution are parallel with respect to this connection. For this, see
Example 3 in Section 3 in \cite{FVV}; also, it is not difficult to deduce it
from the expression for $\nabla $ in the proof of Theorem 1.1 in \cite%
{Claudio}.

\begin{theorem}
\label{subOsc}\emph{a)} $\left( \mathcal{E}^{o},g_{o}\right) $ is a
sub-Lorentzian structure on $\operatorname{Osc}\left( \kappa \right) $ and $\left( 
\mathcal{E}^{\lambda },g_{\lambda }\right) $ is a sub-Riemannian structure or
a sub-Lorentzian structure on $\operatorname{Osc}\left( \kappa \right) $, depending on
whether $\kappa \lambda <0$ or $\kappa \lambda >0$, respectively.

\smallskip

\emph{b)} Via the identification described in Proposition \emph{\ref{identificacion}}
 of $\operatorname{Osc}\left( \kappa \right) $ with
the circle bundle $\mathcal{D}^{\kappa ,1}$, 
a smooth curve $\left( \alpha ,\mathcal{U}\right) $ in $%
\mathcal{D}^{\kappa ,1}$ is admissible for $\mathcal{E}^{o}$ if and only if%
\begin{equation}
a^{\prime }\left( s\right) =f\left( s\right) \mathcal{U}\left( s\right)
+h_{o}\left( s\right) e_{3}  \label{alphaPrimeO}
\end{equation}%
for some functions $f,h_{o}$ and all $s$. In other words, the velocity of
the foot point curve, at any instant $s$, is a linear combination of $%
\mathcal{U}\left( s\right) $ and the vertical direction.

\smallskip

\emph{c)} A smooth curve $\left( a,\mathcal{U}\right) $ in $\mathcal{D}%
^{\kappa ,1}$ is admissible for $\mathcal{E}^{\lambda }$ if and only if the
angular speed of the unit vector $\mathcal{U}$ is proportional, with fixed
ratio $\lambda $, to the vertical rate of change of the foot point curve.
More precisely, putting%
\begin{equation}
a^{\prime }\left( s\right) =W\left( s\right) +h\left( s\right) e_{3}
\label{alphaPrime}
\end{equation}%
with $W\left( s\right) \in \mathcal{D}_{a\left( s\right) }^{\kappa }$ and $%
h\left( s\right) \in \mathbb{R}$, then 
\begin{equation}
\tfrac{D\mathcal{U}}{ds}\left( s\right) =\lambda h\left( s\right) \mathcal{V}%
\left( s\right)  \label{DUdt}
\end{equation}%
for all $s$, where $\frac{D}{ds}$ is the covariant derivative along $a$ for
the adapted connection on $N\left( \kappa \right) $ and the vector field $%
\mathcal{V}$ along $a$ is the one making $\left\{ \mathcal{U}\left( s\right)
,\mathcal{V}\left( s\right) \right\} $ a positively oriented orthonormal
basis for $g_{\lambda }$, for all $s$.

\smallskip

\emph{d)} The geodesics of 
$\left( \operatorname{Osc}\left( \kappa \right),\mathcal{E}^{o},g_{o}\right) $
and of $\left( \operatorname{Osc}\left( \kappa \right),\mathcal{E}^{\lambda },g_{\lambda }\right) $
departing from the identity have the form
\begin{eqnarray}
\gamma \left( s\right) &=&\exp \left( s\phi _{\kappa }\left( u,z,v\right)
\right) \exp \left( s\phi _{\kappa }\left( 0,-yi,0\right) \right) \text{,}
\label{gamma1} \\
\gamma \left( s\right) &=&\exp \left( s\phi _{\kappa }\left( \lambda \left(
v-w\right) ,z,v+w\right) \right) \exp \left( s\phi _{\kappa }\left( \lambda
w,0,-w\right) \right) \text{,}  \label{gamma2}
\end{eqnarray}%
respectively, with $u,v,w\in \mathbb{R}$ and $z=x+iy\in \mathbb{C}$.
\end{theorem}

\begin{remark} \upshape
	As explained above, in this paper we deal only with Hamiltonian geodesics.
We comment in the Appendix (Section \ref{Appendix}) 
that in the sub-Riemannian case they are in
particular locally distance minimizing curves, but, in general, they may not
the only constant speed curves with this property. Those of the theorem
above, however, are the only locally distance minimizing constant speed
curves, since the distribution has step two (see for instance 20.5.1 of \cite%
{AgrachevAbnormal}).
\end{remark}

\noindent \textit{Proof of Theorem \ref{subOsc}. } 
Parts (a) and (d) are very particular cases of Theorem \ref{sub-semi}
below, via Example \ref{Example1} (see the real case in Example \ref{lastexample}). 
Since the proofs are short, we write them
down, as an illustration.

\smallskip

a) Let $b_{\kappa }$ be the quadratic form on $\operatorname{osc}\left( \kappa
\right) $ associated with the Lorentz metric (\ref{biOsc}). We compute%
\begin{equation*}
b_{\kappa }\left( \phi \left( u,x,0,v\right) \right) =x^{2}+2\kappa uv\text{%
\ \ \ \ and\ \ \ \ \ }b_{\kappa }\left( \phi \left( \lambda v,x,y,v\right)
\right) =x^{2}+y^{2}+2\kappa \lambda v^{2}\text{.}
\end{equation*}%
Hence, the metric on $\mathcal{E}_{I_{4}}^{o}$ is Lorentzian and that on $%
\mathcal{E}_{I_{4}}^{\lambda }$ is Lorentzian or Riemannian depending on
whether $\kappa \lambda <0$ or $\kappa \lambda >0$, respectively. It remains
only to check the bracket generating condition. We do it next. Notice that $%
U,X$ as in (\ref{UXYV}) are in $\mathcal{E}^{o}$ and by (\ref{bracketUV}), $%
\left[ U,X\right] =Y$, which spans a complement of $\mathcal{E}^{o}$.
Similarly, $X$ and $Y$ are in $\mathcal{E}^{\lambda }$ and $\left[ X,Y\right]
=\kappa U,$ which spans a complement of $\mathcal{E}^{\lambda }$, since $%
\lambda \neq 0.$

\smallskip

b) Let $\left( a,\mathcal{U}\right) $ be a smooth curve in $\mathcal{D}%
^{\kappa ,1}$. Under the identification in Proposition \ref{identificacion},
it corresponds to the curve $\sigma \left( s\right) =\left( 
\begin{array}{cc}
1 & 0 \\ 
a\left( s\right) & A\left( s\right)%
\end{array}%
\right) $ in $GL\left( 4,\mathbb{R}\right) $, where $A\left( s\right) $ is
the unique matrix such that $\sigma \left( s\right) \in \operatorname{Osc}\left(
\kappa \right) $ and $A\left( s\right) e_{1}=\mathcal{U}\left( s\right) $
for all $s$. The condition for $\sigma $ to be admissible for $\mathcal{E}%
^{o}$ is that%
\begin{equation}
\left( 
\begin{array}{cc}
1 & 0 \\ 
a & A%
\end{array}%
\right) ^{-1}\left( 
\begin{array}{cc}
1 & 0 \\ 
a^{\prime } & A^{\prime }%
\end{array}%
\right) \in \mathcal{E}_{I_{4}}^{o}  \label{Ainv}
\end{equation}%
If $a$ and $A$ are as in (\ref{aA}), this translates into 
\begin{equation*}
\left( x^{\prime },y^{\prime }\right) =\xi \left( \cos u,\sin u\right) \text{
\ \ \ and \ \ \ }v^{\prime }=\nu  -\tfrac{1}{2}\kappa \xi \left( y\cos
u-x\sin u\right)
\end{equation*}%
for some functions $\nu  ,\xi $ (it is clear that $u^{\prime }=\mu $ for
some function $\mu $). On the other hand, condition (\ref{alphaPrimeO}) can
be written as 
\begin{equation*}
p\left( a^{\prime }\left( s\right) -f\left( s\right) A\left( s\right)
e_{1}\right) =0\text{,}
\end{equation*}%
where $p:\mathbb{R}^{3}\rightarrow \mathbb{R}^{2}$ is the canonical
projection onto the first two coordinates. The left-hand side is $\left(
x^{\prime }\left( s\right) ,y^{\prime }\left( s\right) \right) -f\left(
s\right) \left( \cos u\left( s\right) ,\sin u\left( s\right) \right) $.
Thus, if $\sigma $ is admissible for $\mathcal{E}^{o}$, one can take $f=\xi $
and an appropriate function $h$ to ensure that condition (\ref{alphaPrimeO})
for $\left( a,\mathcal{U}\right) $ is satisfied. Conversely, one can take $%
\xi =f$ and a suitable function $\nu  $.

\smallskip

c) We look at condition (\ref%
{Ainv}), but now for $\mathcal{E}^{\lambda }$ instead of $\mathcal{E}^{o}$ (also for $%
a$ and $A$ as in (\ref{aA})). 
By multiplying the corresponding matrices, we obtain that 
it is equivalent to the existence of smooth
functions $\xi ,\eta ,\nu  $ such that%
\begin{eqnarray*}
x^{\prime } &=&\xi \cos u-\eta \sin u\text{,\ \ \ \ }y^{\prime }=\eta \cos
u+\xi \sin u\text{;} \\
v^{\prime } &=&\nu  -\tfrac{1}{2}\kappa \left( y\left( \xi \cos u-\eta
\sin u\right) -x\left( \xi \sin u+\eta \cos u\right) \right) \text{;} \\
u^{\prime }\cos u &=&\lambda \nu  \cos u\text{,\ \ \ \ }u^{\prime }\sin
u=\lambda \nu  \sin u\text{.}
\end{eqnarray*}%
This, in turn, can be simplified as%
\begin{eqnarray}
\left( x+iy\right) ^{\prime } &=&\left( \xi +i\eta \right) e^{iu}\text{;}
\label{uno} \\
v^{\prime } &=&\nu  -\tfrac{1}{2}\kappa \left( yx^{\prime }-xy^{\prime
}\right) \text{;}  \label{tres} \\
u^{\prime }e^{iu} &=&\lambda \nu  e^{iu}\text{.}  \label{DOS}
\end{eqnarray}

On the other hand, $\mathcal{U}\left( s\right) =A\left( s\right) e_{1}=\cos
\left( u\left( s\right) \right) \mathcal{X}_{a\left( s\right) }+\sin \left(
u\left( s\right) \right) \mathcal{Y}_{a\left( s\right) }$ and so, 
\begin{equation*}
\mathcal{V}\left( s\right) =A\left( s\right) e_{2}=-\sin \left( u\left(
s\right) \right) \mathcal{X}_{a\left( s\right) }+\cos \left( u\left(
s\right) \right) \mathcal{Y}_{a\left( s\right) }\text{.}
\end{equation*}%
Since $\mathcal{X}$ and $\mathcal{Y}$ are parallel for the adapted
connection on $N\left( \kappa \right) $, we have $\frac{D\mathcal{U}}{ds}%
=u^{\prime }\mathcal{V}$. Besides, 
\begin{equation*}
a^{\prime }=\left( x^{\prime },y^{\prime },v^{\prime }\right) =x^{\prime
}\left( 1,0,-\tfrac{1}{2}\kappa y\right) +y^{\prime }\left( 0,1,\tfrac{1}{2}%
\kappa x\right) +\left( 0,0,h\right) \text{,}
\end{equation*}%
with $h$ as in (\ref{alphaPrime}) given by $h=v^{\prime }-%
\frac{1}{2}\kappa \left( xy^{\prime }-yx^{\prime }\right) $.

Now, suppose that $\sigma $ is admissible for $\mathcal{E}^{\lambda }$. By (%
\ref{DOS}), $u^{\prime }=\lambda \nu  $. By (\ref{tres}), $\nu 
=v^{\prime }-\frac{1}{2}\kappa \left( xy^{\prime }-yx^{\prime }\right) =h$.
Consequently, $\frac{D\mathcal{U}}{ds}=u^{\prime }\mathcal{V}=\lambda
\nu  \mathcal{V}=\lambda h\mathcal{V}$, as desired.

Conversely, if (\ref{alphaPrime}) and (\ref{DUdt}) hold, then one can choose 
$\nu  $ satisfying (\ref{tres}), then $h=\nu  $ and so (\ref{DOS})
holds. Finally, by taking $\xi +i\eta =\left( x+iy\right) ^{\prime }e^{-iu}$
we are ready with (\ref{uno}) and so we conclude that $\sigma $ is
admissible for $\mathcal{E}^{\lambda }$.

\smallskip

d) As mentioned at the beginning of the proof, this is a very particular
case of a result below. It is a direct application of Proposition \ref%
{PropRM} and we present it as an example. It is required that the metrics $%
g_{o}$ and $g_{\lambda }$ are nondegenerate, what we know from (a). Also, we
compute, using (\ref{biOsc}), the orthogonal complements of $\mathcal{E}^{o}$
and $\mathcal{E}^{\lambda }$ at the identity, obtaining%
\begin{equation}
\mathfrak{h}^{o}=\left\{ \phi _{\kappa }\left( 0,0,y,0\right) \mid y\in 
\mathbb{R}\right\} \text{ \ \ \ and\ \ \ \ }\mathfrak{h}^{\lambda }=\left\{
\phi _{\kappa }\left( -\lambda \nu  ,0,0,\nu  \right) \mid \nu 
\in \mathbb{R}\right\} \text{,}  \label{hohlambda}
\end{equation}%
respectively. Since they are clearly subalgebras, formulas (\ref{gamma1})
and (\ref{gamma2}) hold. \hfill $\square$

\bigskip

The distributions $\mathcal{E}^{o}$ and $\mathcal{E}^{\lambda }$ for $%
\lambda \neq 0$ are, up to isomorphism, those for which Proposition \ref%
{PropRM} below can be applied. In fact, we looked at the classification of
the Lie subalgebras of the oscillator Lie algebras in a corollary of Theorem
1 in \cite{Biggs} and found that $\mathfrak{h}^{o}$ and $\mathfrak{h}%
^{\lambda }$ as in (\ref{hohlambda}) are the only ones (up to isomorphism)
that are nondegenerate and whose orthogonal complements are bracket
generating.

\medskip

In the following proposition we restrict ourselves to the case $\kappa =1$.
For simplicity we sometimes write $\Phi $ and $\phi $ instead of $\Phi _{1}$
and $\phi _{1}$, respectively. We find which of the geodesics in the theorem
above are periodic. Notice that $\operatorname{Osc}\left( 1\right) $ is
diffeomorphic to $S^{1}\times \mathbb{R}^{3}$.

\begin{proposition}
\label{PropoPeriodic}\emph{a)} The periodic geodesics of $\left( \operatorname{Osc}%
\left( 1\right) ,\mathcal{E}^{o}\right) $ through the identity are exactly
those of the form%
\begin{equation}
\gamma \left( s\right) =\Phi _{1}\left( us,ic\left( 1-e^{ius}\right) ,-%
\tfrac{1}{2}c^{2}\sin \left( us\right) \right)  \label{periodic1}
\end{equation}%
for some $u,c\in \mathbb{R}$.

\smallskip

\emph{b)} Let $\lambda \neq 0$, as above. The nonconstant periodic geodesics of $%
\left( \operatorname{Osc}\left( 1\right) ,\mathcal{E}^{\lambda }\right) $ through
the identity are exactly those given by%
\begin{equation}
\gamma \left( s\right) =\Phi _{1}\left( \lambda urs,\tfrac{1}{\lambda u}%
\left( 1-e^{i\lambda us}\right) iz,\tfrac{r}{\lambda }\sin \left( \lambda
us\right) \right) \text{,}  \label{periodic2}
\end{equation}%
for any $u\neq 0$, with $z\in \mathbb{C}$ and $r\in \mathbb{Q}$ satisfying $%
\left\vert z\right\vert ^{2}+2\lambda ru^{2}=0$.

\smallskip

\emph{c)} Identifying $\gamma $ with a curve $\left( \alpha ,\mathcal{U}%
\right) $ in $\mathcal{D}^{1,1}$ as in Proposition \ref{identificacion}, the
foot point curve $\alpha $ describes an ellipse in $\mathbb{R}^3 $
contained in the plane through the origin orthogonal to the vectors $\left( c,2\right) $
and $\left( z,2\lambda u\right) $, for \emph{(a)} and \emph{(b)}, respectively.
\end{proposition}

Before the proof we note that if we identify $\mathbb{R}^{2}\equiv \mathbb{C}
$ in the usual way, and write $\operatorname{Osc}\left( \kappa \right) $ $=\Phi
_{\kappa }\left( \mathbb{R\times C}\times \mathbb{R}\right) $ as above, then
a straightforward computation yields:%
\begin{equation}
\Phi _{\kappa }\left( u,z,v\right) \Phi _{\kappa }\left( \mu ,\zeta
,\nu  \right) =\Phi _{\kappa }\left( u+\mu ,z+e^{iu}\zeta ,v+\nu  +%
\tfrac{1}{2}\kappa \det \left( z,e^{iu}\zeta \right) \right) \text{,}
\label{multipliPhi}
\end{equation}%
where $\det \left( x+iy,c+id\right) =xd-yc$.

\bigskip

\noindent \textit{Proof of Proposition \ref{PropoPeriodic}. } 
a) By Theorem \ref{subOsc} (d), the geodesics through the
identity have the form $\gamma \left( s\right) =\gamma _{1}\left( s\right)
\gamma _{2}\left( s\right) $, with $\gamma _{1}\left( s\right) =\exp \left(
s\phi \left( u,z,v\right) \right) $ and $\gamma _{2}\left( s\right) =\exp
\left( s\phi \left( 0,-yi,0\right) \right) $ for $u,v\in \mathbb{R}$ and $%
z=x+iy$. By (\ref{gammaPropo1}), $\gamma _{2}\left( s\right) =\Phi \left(
0,-iys,0\right) $.

Now we study the first factor and consider two cases.

If $u=0$, then by (\ref{gammaPropo1}), $\gamma _{1}\left( s\right) =\Phi
\left( 0,sz,vs\right) $. We multiply $\gamma _{1}\left( s\right) $ and $%
\gamma _{2}\left( s\right) $ using (\ref{multipliPhi}) and obtain $\gamma
\left( s\right) =\Phi \left( 0,xs,vs-\frac{1}{2}xys^{2}\right) $. So, in
this case, only the constant geodesic $\Phi \left( 0\right) $ is periodic.

If $u\neq 0$, again by (\ref{gammaPropo1}),%
\begin{equation}
\gamma _{1}\left( s\right) =\Phi \left( us,\tfrac{1}{u}\left(
1-e^{ius}\right) iz,vs+\tfrac{1}{2u^{2}}\left\vert z\right\vert ^{2}\left(
us-\sin \left( us\right) \right) \right) \text{.}  \label{gamma1Prueba}
\end{equation}%
By (\ref{multipliPhi}), $\gamma \left( s\right) =\Phi \left( us,\mathcal{Z}%
\left( s\right) ,\mathcal{W}\left( s\right) \right) $, with $\mathcal{Z}%
\left( s\right) =\tfrac{1}{u}\left( 1-e^{isu}\right) iz-e^{isu}iys$ and some
function $\mathcal{W}$. Assume that $\gamma $ is periodic. Then $\mathcal{Z}$
is periodic and so $y=0$. In particular, $\gamma _{2}=I_{4}$ is constant.
Hence, $\gamma =\gamma _{1}$, which we describe next. We have that $z=x\in 
\mathbb{R}$. The periodicity of $\mathcal{W}$ implies that $v+\tfrac{1}{2u}%
x^{2}=0$ (see \ref{gamma1Prueba}). Thus,%
\begin{equation*}
\gamma \left( s\right) =\Phi \left( us,\tfrac{1}{u}\left( 1-e^{isu}\right)
ix,-\tfrac{1}{2u^{2}}x^{2}\sin \left( us\right) \right) \text{.}
\end{equation*}%
Since $x\in \mathbb{R}$ is arbitrary, we call $\frac{x}{u}=c$ and obtain
expression (\ref{periodic1}).

Such $\gamma $ is clearly a geodesic and is periodic, since $\Phi \left(
u,z,v\right) =\Phi \left( u+2k\pi ,z,v\right) $ for all $u,v\in \mathbb{R}$,
 $z\in \mathbb{C}$.

\medskip

b) By (\ref{gamma2}), the geodesics have the form $\gamma \left( s\right)
=\gamma _{1}\left( s\right) \gamma _{2}\left( s\right) $, with 
\begin{equation*}
\gamma _{1}\left( s\right) =\exp \left( s\phi \left( \lambda \left(
v-w\right) ,z,v+w\right) \right) \text{\ \ \ and\ \ \ }\gamma _{2}\left(
s\right) =\exp \left( s\phi \left( \lambda w,0,-w\right) \right) \text{.}
\end{equation*}
By (\ref{gammaPropo1}), $\gamma _{2}\left( s\right) =\Phi \left( \lambda
ws,0,-ws\right) $. For $\gamma _{1}$ we consider two cases:

If $v=w$, by (\ref{gammaPropo1}), we have $\gamma _{1}\left( s\right) =\Phi
\left( 0,sz,2ws\right) $. Using (\ref{multipliPhi}) we compute $\gamma
\left( s\right) =\Phi \left( \lambda ws,sz,ws\right) $. If it is periodic,
then $z=w=0$. Hence, in this case only the constant geodesic $\Phi \left(
0\right) $ is periodic.

Now suppose that $v\neq w$. We call $u=v-w$; in particular, $v=u+w$. By (\ref%
{gammaPropo1}), we have
\begin{equation*}
\gamma _{1}\left( s\right) =\Phi \left( \lambda us,\tfrac{1}{\lambda u}%
\left( 1-e^{i\lambda us}\right) iz,\left( v+w\right) s+\tfrac{\kappa }{%
2\lambda ^{2}u^{2}}\left\vert z\right\vert ^{2}\left( \lambda us-\sin \left(
\lambda us\right) \right) \right)
\end{equation*}%
We multiply it with $\Phi \left( \lambda ws,0,-ws\right) $ using (\ref%
{multipliPhi}) and get%
\begin{equation*}
\gamma \left( s\right) =\Phi \left( \lambda vs,\tfrac{1}{\lambda u}\left(
1-e^{i\lambda us}\right) iz,T\left( s\right) \right) \text{,}
\end{equation*}%
where the first coordinate was obtained as $\lambda us+\lambda ws$, and 
\begin{eqnarray*}
T\left( s\right) &=&\left( v+w\right) s+\tfrac{\left\vert z\right\vert ^{2}}{%
2\lambda ^{2}u^{2}}\left( \lambda us-\sin \left( \lambda us\right) \right)
-ws \\
&=&\left( v+\tfrac{\left\vert z\right\vert ^{2}}{2\lambda u}\right) s-\tfrac{%
\left\vert z\right\vert ^{2}}{2\lambda ^{2}u^{2}}\sin \left( \lambda
us\right) \text{.}
\end{eqnarray*}%
Hence, if $\gamma $ is periodic, then $v+\tfrac{\left\vert z\right\vert ^{2}}{%
2\lambda u}=0$ and $e^{i\lambda us}$ and $e^{i\lambda vs}$ must be
simultaneously periodic. This means that $2\lambda uv+\left\vert
z\right\vert ^{2}=0$ and that there exists $s_{o}$ such that $e^{i\lambda
us}=e^{i\lambda u\left( s+s_{o}\right) }$ and $e^{i\lambda vs}=e^{i\lambda
v\left( s+s_{o}\right) }$ for all $s$. The latter implies that $\lambda
us_{o}$ and $\lambda vs_{o}$ are integral multiples of $2\pi $, and so, $%
v=ru $ for some rational number $r$. Consequently, $\lambda vs=\lambda rus$
and $-\left\vert z\right\vert ^{2}=2\lambda uv=2\lambda ru^{2}$. Thus, $%
T\left( s\right) =\tfrac{r}{\lambda }\sin \left( \lambda us\right) $ and so $%
\gamma $ is as stated in (\ref{periodic2}).

Conversely, we observe, on the one hand, that the curve $\gamma $ is
periodic, and on the other hand, that it is a geodesic for $\mathcal{E}%
^{\lambda }$, since a lengthy but straightforward computation shows that it
is equal to the product in (\ref{gamma2}), with $v=ru$ and $w=\left(
r-1\right) u$ for some $u\neq 0$.

\smallskip

c) Lengthy computations show that in each case, $\alpha $ is contained in the
corresponding plane. The trajectory is an ellipse since it projects to a
circle. \hfill $\square$

\medskip

\begin{remark} \upshape
The proof of the proposition shows that the periodic sub-Lorentzian geodesics 
$\left( \operatorname{Osc}\left( 1\right) ,\mathcal{E}^{o},g_{o}\right) $ in 
(a) turn out to be geodesics of $\operatorname{Osc}\left( 1\right) $ for the
bi-invariant metric, admissible for $\mathcal{E}^{o}$. So, in a sense, they
are not new, since all the periodic geodesics of a quotient $\Gamma
\backslash \operatorname{Osc}\left( 1\right) $ have been described in \cite{Gabriela}%
; in our case, it is not even necessary to pass to the quotient. Thus, our
contribution in part (a) is that there no other periodic geodesics.

In contrast, part (b), for $\mathcal{E}^{\lambda }$, features so to
say genuine periodic sub-Lorentzian or sub-Riemannian geodesics.
\end{remark}

\section{Higher dimensional analogs\label{SectionHD}}

The oscillator Lie algebras are among the simplest double extensions of Lie
algebras, introduced in \cite{Medina, Favre} (see also 
\cite{OvandoTorino}), which is closely linked to the existence of
bi-invariant metrics. In this section we consider a particular case of
solvable Lie algebras obtained by this procedure. We start directly from
a nilpotent Lie group, omitting the first step, an extension of an abelian subalgebra.

\subsection{Nilpotent groups associated with semisimple symmetric pairs and
their sub-Riemannian geodesics\label{HighNilpotent}}

The higher dimensional analogs of the oscillator groups will be extensions
of nilpotent groups associated with semisimple symmetric pairs. We review
them next.

Let $\left( G,K\right) $ be a semisimple symmetric pair, that is, $G$ is a
semisimple Lie group and $K$ is a closed subgroup of $G$ such that $F\left(
\vartheta \right) _{o}\subset K\subset F\left( \vartheta \right) $ for some
involutive automorphism $\vartheta $ of $G$, where $F\left( \vartheta \right)
=\left\{ g\in G\mid \vartheta \left( g\right) =g\right\} $ and the subindex $o$
means taking the identity component ($\vartheta $ is uniquely determined by $K$%
). We assume that $K$ is compact.

Let $\mathfrak{g}$ and $\mathfrak{k}$ be the Lie algebras of $G$ and $K$,
respectively, and let $\mathfrak{p}$ $=\left\{ z\in \mathfrak{g}\mid d\vartheta
_{e}\left( z\right) =-z\right\} $ ($e$ denotes the identity of $G$). Then $%
\mathfrak{g}=\mathfrak{p}\oplus \mathfrak{k}$ is the Cartan decomposition of 
$\mathfrak{g}$ associated with $\mathfrak{k}$. Let $B$ one half of the
Killing form of $G$ (with this normalization, 
the mapping $F_\kappa$ in Proposition \ref{13} below is an isometry). 
We suppose additionally that the pair $\left(
G,K\right) $ is either of the compact or the noncompact type, that is, $%
\left. B\right\vert _{\mathfrak{p}\times \mathfrak{p}}$ is negative or
positive definite, respectively. We call them $\kappa $-types ($\kappa =1$
and $\kappa =-1$ for the compact and noncompact type, respectively). Let $%
G/K $ be the associated symmetric space, with the Riemannian metric induced
by the inner product $-\kappa B$ restricted to $\mathfrak{p}\equiv
T_{K}\left( G/K\right) $. Thus, $G/K$ has nonnegative or nonpositive
sectional curvature depending on whether it is of $1$-type or $\left(
-1\right) $-type.

Let $\mathfrak{n}\left( \mathfrak{g},\mathfrak{k}\right) $ be the vector
space $\mathfrak{g}$ endowed with the Lie bracket $\left[ \left( z,v\right)
,\left( \zeta ,\nu  \right) \right] =\left( 0,\left[ z,\zeta \right] _{%
\mathfrak{g}}\right) $, for $z,\zeta \in \mathfrak{p}$, $v,\nu  \in 
\mathfrak{k}$. It is the Lie algebra of the two-step nilpotent Lie group $N\left( 
\mathfrak{g},\mathfrak{k}\right) $, which is again the
vector space $\mathfrak{g}$, now with multiplication 
\begin{equation}
\left( z,v\right) \cdot \left( \zeta ,\nu  \right) =\left( z+\zeta
,v+\nu  +\tfrac{1}{2}\left[ z,\zeta \right] _{\mathfrak{g}}\right) \text{%
.}\label{multNgk}
\end{equation}

Let $\mathcal{D}$ be the left-invariant distribution on $N\left( \mathfrak{%
g,k}\right) $ determined by $\mathcal{D}_{0}=\mathfrak{p}$, that is, 
\begin{equation}
\mathcal{D}_{\left( z,v\right) }=\text{graph}\left( \tfrac{1}{2}\left. \operatorname{ad}_{z}\right\vert _{\mathfrak{p}}\right)
=\left\{\left(\zeta,\tfrac{1}{2}\left[z,\zeta\right]\right)\mid \zeta\in\mathfrak{p}\right\}.  
\label{graph}
\end{equation}%
The restriction of $-\kappa B$ to $\mathfrak{p}\equiv \mathcal{D}_{\left(
0,0\right) }$ defines a sub-Riemannian structure on $\left( N\left( 
\mathfrak{g},\mathfrak{k}\right) ,\mathcal{D}\right) $, since it is well-known that $%
\left[ \mathfrak{p},\mathfrak{p}\right] =\mathfrak{k}$.

We will see below in Example \ref{Example1} that, setting
$\mathfrak{g}_{1}=\mathfrak{o}\left( 3\right) $ and $%
\mathfrak{g}_{-1}=\mathfrak{o}\left( 1,2\right) $, then
 $\left( N\left( \mathfrak{g}%
_{\kappa },\mathfrak{o}\left( 2\right) \right) ,\mathcal{D}\right) $ is isomorphic to $%
\left( N\left( \kappa \right) ,\overline{\mathcal{D}}^{\kappa }\right)
\equiv \left( \operatorname{Hei}\left( \kappa \right) ,\mathcal{D}^{\kappa }\right) $
for $\kappa =\pm 1$. As we mentioned in the introduction,
apart from the Heisenberg group, the most studied group of the type $%
N\left( \mathfrak{g},\mathfrak{k}\right) $ is for $\mathfrak{g}=\mathfrak{o}\left( 1+m\right) $%
, $\mathfrak{k}$ $=\mathfrak{o}\left( m\right) $, that is, the free, two-step Carnot
group in $m$ generators, whose sub-Riemannian geometry has been understood
in \cite{RizziCut} (see also \cite{Myani} and \cite{AnnamriaCruz})

\bigskip

In order to find expressions for the sub-Riemannian geodesics of $N\left( 
\mathfrak{g},\mathfrak{k}\right) $ we introduce some notation.

\begin{notation}
Given $u\in \mathfrak{k}$, let $h_{u}:\mathbb{R}\rightarrow \operatorname{End}\left( 
\mathfrak{g}\right) $ be defined by $h_{u}\left( s\right) =\int_{0}^{s}\operatorname{%
Ad}_{G}\left( e^{tu}\right) ~dt$.
\end{notation}

Notice that $h_{u}\left( s\right) $ preserves $\mathfrak{k}$ and $\mathfrak{p%
}$ for all $s$. For convenience, we present another expression for the
function $h_{u}$. As usual, given an analytic function $f$ and a linear
operator $T$, $f\left( T\right) $ is defined by $\sum_{n=0}^{\infty }\frac{%
f^{\left( n\right) }\left( 0\right) }{n!}T^{n}$.

\begin{proposition}
Let $f:\mathbb{R}\rightarrow \mathbb{R}$ be the analytic function such that $%
f\left( t\right) =\frac{1}{t}\left( e^{t}-1\right) $ for $t\neq 0$; in
particular, $f\left( 0\right) =1$. Then 
$h_{u}\left( s\right) =sf\left( s\operatorname{ad}_{u}\right)$.
\end{proposition}
\begin{proof}
The statement is well-known; we check its validity for the sake of
completeness: We call $\psi \left( s\right) $ the right-hand site. Since $%
h_{u}\left( 0\right) =\psi \left( 0\right) =0$, it suffices to show that $%
\operatorname{Ad}\left( e^{su}\right) =\psi ^{\prime }\left( s\right) $ for all $s$.
We compute%
\begin{equation*}
\psi ^{\prime }\left( s\right) =f\left( s\operatorname{ad}_{u}\right)
+sf^{\prime }\left( s\operatorname{ad}_{u}\right) \operatorname{ad}
_{u}=f\left( s\operatorname{ad}_{u}\right) +e^{s\operatorname{ad}
_{u}}-f\left( s\operatorname{ad}_{u}\right) =\operatorname{Ad}\left( e^{su}\right) 
\text{,}
\end{equation*}%
since $f^{\prime }\left( t\right) =\frac{1}{t}\left( e^{t}-f\left( t\right)
\right) $.
\end{proof}

For a vector $W\in \mathfrak{g}$, we denote by $W^{\flat }\in \mathfrak{g}%
^{\ast }$ the covector defined by $W^{\flat }\left( Z\right) =-\kappa
B\left( W,Z\right) $.

\begin{theorem}
\label{geoSubN}Let $z\in \mathfrak{p}$ and $u\in \mathfrak{k}$. The
sub-Riemannian geodesic $\sigma $ in $\left( N\left( \mathfrak{g},\mathfrak{k}\right) ,%
\mathcal{D}\right) $ with $\sigma \left( 0\right) =0$ and initial covector $%
\left( z+u\right) ^{\flat }$ is given explicitly by%
\begin{equation}
\sigma _{\left( z,u\right) }\left( s\right) =\left( h_{u}\left( s\right)
\left( z\right) ,\tfrac{1}{2}\int_{0}^{s}\left[ h_{u}\left( t\right) \left(
z\right) ,h_{u}^{\prime }\left( t\right) \left( z\right) \right] ~dt\right) 
\text{.}  \label{sigmaXU}
\end{equation}
\end{theorem}

\begin{proof} For simplicity we write $\mathfrak{n}=\mathfrak{n}\left( 
 \mathfrak{g},\mathfrak{k}\right) $ and $N=N\left( \mathfrak{g},\mathfrak{k}\right) $. In the
Appendix (Section \ref{Appendix}) we recall the definitions of geodesics for a
general Hamiltonian and of the Hamiltonian associated with a sub-Riemannian
metric. In our case, the sub-Riemannian structure on $N$ is the
left-invariant one induced by $b:\mathfrak{n}^{\ast }\rightarrow \mathfrak{n}
$, $b\left( \left( z,v\right) ^{\flat }\right) =\left( z,0\right) $ and
 the associated Hamiltonian is the left-invariant
function $H:T^{\ast }N\equiv N\times \mathfrak{n}^{\ast }$, $H\left(
p,\theta \right) =\frac{1}{2}\theta \left( b\left( \theta \right) \right) $.

It is well-known (see, for instance, Section 7.3 of the survey paper \cite%
{SachkovSurveyLie}) that under these conditions, the sub-Riemannian geodesic 
$\sigma $ is the unique smooth curve in $N$ such that the curve $\left(
\sigma ,\theta \right) $ in $T^{\ast }N$ satisfies the equations%
\begin{equation}
\sigma ^{\prime }\left( s\right) =dL_{\sigma \left( s\right) }\left( b\left(
\theta \left( s\right) \right) \right) \text{\ \ \ \ \ \ \ and \ \ \ \ \ }%
\theta ^{\prime }\left( s\right) =\left( \operatorname{ad}_{b\left( \theta
\left( s\right) \right) }\right) ^{\ast }\theta \left( s\right)
\label{sigmaPalphaP}
\end{equation}%
for all $s$, with initial conditions $\sigma \left( 0\right) =\left(
0,0\right) $ and $\theta \left( 0\right) =\left( z+u\right) ^{\flat }$ ($%
L_{g}$ denotes the left multiplication by $g$). Putting $\theta =\left(
Z,U\right) ^{\flat }$, with $Z,U$ curves in $\mathfrak{p}$ and $\mathfrak{k}$%
, respectively, the equations translate into%
\begin{eqnarray*}
\sigma ^{\prime }\left( s\right) &=&dL_{\sigma \left( s\right) }\left(
Z\left( s\right) ,0\right) \text{,} \\
B\left( \left( Z^{\prime }\left( s\right) ,U^{\prime }\left( s\right)
\right) ,\left( \zeta ,\mu \right) \right) &=&B\left( \left( Z\left(
s\right) ,U\left( s\right) \right) ,\left[ \left( Z\left( s\right) ,0\right)
,\left( \zeta ,\mu \right) \right] _{\mathfrak{n}}\right)
\end{eqnarray*}%
for all $s$ and all $\zeta \in \mathfrak{p}$, $\mu \in \mathfrak{k}$. The
second expression is equivalent to%
\begin{eqnarray*}
B\left( Z^{\prime }\left( s\right) ,\zeta \right) +B\left( U^{\prime }\left(
s\right) ,\mu \right) &=&B\left( \left( Z\left( s\right) ,U\left( s\right)
\right) ,\left( 0,\left[ Z\left( s\right) ,\zeta \right] _{\mathfrak{g}%
}\right) \right) \\
&=&B\left( U\left( s\right) ,\left[ Z\left( s\right) ,\zeta \right] _{%
\mathfrak{g}}\right) \\
&=&B\left( \left[ U\left( s\right) ,Z\left( s\right) \right] _{\mathfrak{g}%
},\zeta \right) \text{.}
\end{eqnarray*}

Since $B$ is nondegenerate, we have that $U^{\prime }$ vanishes, and hence $%
U=u$ (constant) and $Z^{\prime }=\left[ u,Z\right] $, which yields $Z\left(
s\right) =e^{s\operatorname{ad}_{u}}z$ for all $s$.

Now we are in the position of checking that $\sigma _{\left( z,u\right) }$
in (\ref{sigmaXU}) satisfies the first condition in (\ref{sigmaPalphaP}). We
call $C\left( \tau \right) =\left[ h_{u}\left( \tau \right) z,\operatorname{Ad}%
\left( e^{\tau u}\right) z\right] \in \left[\mathfrak{p},\mathfrak{p}\right]
=\mathfrak{k}$. Using (\ref{multNgk}) 
we compute
\begin{eqnarray*}
dL_{\sigma \left( s\right) }\left( Z\left( s\right) ,0\right) &=&\left. 
\frac{d}{dt}\right\vert _{0}\left( h_{u}\left( s\right)
z,\int_{0}^{s}C\left( \tau \right) ~d\tau \right) \left( te^{s\operatorname{ad}%
_{u}}z,0\right) \\
&=&\left. \frac{d}{dt}\right\vert _{0}\left( h_{u}\left( s\right) z+t\operatorname{%
Ad}\left( e^{su}\right) z,\int_{0}^{s}C\left( \tau \right) ~d\tau +\tfrac{1}{%
2}tC\left( s\right) \right) \\
&=&\left( h_{u}^{\prime }\left( s\right) z,\tfrac{1}{2}C\left( s\right)
\right) =\sigma _{\left( z,u\right) }^{\prime }\left( s\right) \text{,}
\end{eqnarray*}%
as we wanted to show.
\end{proof} 

\subsection{The oscillator groups $\operatorname{Osc}\left( G,K\right) $ for
semisimple symmetric pairs $\left( G,K\right) $}

Next we consider $\operatorname{Ad}_{G}\left( K\right) $, which is a compact
subgroup of $GL\left( \mathfrak{g}\right) $. Its elements preserve $%
\mathfrak{p}$ and $\mathfrak{k}$. Notice that in general $\operatorname{Ad}%
_{G}\left( K\right) $ is not equal to $\operatorname{Ad}_{K}\left( K\right) $, and
that $K$ may be not semisimple, even if $G$ is simple. In Examples \ref{Example1}
and \ref{Example2} below, $\operatorname{Ad}_{K}\left( K\right) $ is trivial, but
the kernel of $\left. \operatorname{Ad}_{G}\right\vert _{K}$ is trivial or
isomorphic to $\mathbb{Z}_{2}$, respectively.

\begin{definition}
The group $\operatorname{Osc}\left( G,K\right) $ is $\operatorname{Ad}_{G}\left( K\right)
\times N\left( \mathfrak{g},\mathfrak{k}\right) $ equipped with the multiplication 
\begin{equation}
\left( k,z,v\right) \cdot \left( h,\zeta ,\nu \right) =\left( kh,z+k\left(
\zeta \right) ,v+k\left( \nu \right) +\tfrac{1}{2}\left[ z,k\left( \zeta
\right) \right] \right) \text{.}  \label{multiplicationOGK}
\end{equation}
\end{definition}

We show below that the Lie algebra of $\operatorname{Osc}\left( G,K\right) $ is $%
\operatorname{osc}\left( \mathfrak{g},\mathfrak{k}\right) 
=\operatorname{ad}_{\mathfrak{g}%
}\left( \mathfrak{k}\right) \times \mathfrak{n}\left( \mathfrak{g},\mathfrak{k}\right)$, 
with the Lie bracket%
\begin{equation}
\left[ \left( \operatorname{ad}_{u},z,v\right) ,\left( \operatorname{ad}_{\mu },\zeta
,\nu  \right) \right] =\left( \operatorname{ad}_{\left[ u,\mu \right] },\left[
u,\zeta \right] -\left[ \mu ,z\right] ,\left[ z,\zeta \right] +\left[
u,\nu  \right] -\left[ \mu ,v\right] \right)  \label{bracketHD}
\end{equation}%
(notice that $\operatorname{ad}_{\mathfrak{g}}:\mathfrak{g}\rightarrow gl\left( 
\mathfrak{g}\right) $ is injective, since a semisimple Lie algebra is
centerless). This is a double extension as mentioned at the beginning of the
section. The literature covers mostly only the Lie algebra level; we have
not found an expression for the group multiplication, so we give some
details in the following proposition. For the sake of simplicity, henceforth
we sometimes write $\operatorname{Ad}$ instead of $\operatorname{Ad}_{G}$.

\begin{proposition}
$\operatorname{Osc}\left( G,K\right) $ is in fact a Lie group and its Lie algebra is 
$\operatorname{osc}\left( \mathfrak{g},\mathfrak{k}\right) $.
\end{proposition}

\begin{proof} One option is to write the multiplication as a composition of
semi direct products. Instead, we resort to matrices. 
The group $\operatorname{Osc}\left( G,K\right) $ will end up being the group of
affine transformations of $\mathfrak{p}\oplus \mathfrak{k}$ of the form%
\begin{equation*}
\left( 
\begin{array}{c}
\zeta \\ 
\nu %
\end{array}%
\right) \mapsto \left( 
\begin{array}{c}
z \\ 
v%
\end{array}%
\right) +\left( 
\begin{array}{cc}
\left. k\right\vert _{\mathfrak{p}} & 0 \\ 
\frac{1}{2}\operatorname{ad}_{z}\circ \left. k\right\vert _{\mathfrak{p}} & 
\left. k\right\vert _{\mathfrak{k}}%
\end{array}%
\right) \left( 
\begin{array}{c}
\zeta \\ 
\nu %
\end{array}%
\right)
\end{equation*}%
for some $k\in \operatorname{Ad}\left( K\right) $, $z\in \mathfrak{p}$, $v\in 
\mathfrak{k}$. Indeed, straightforward computations show that $F:\operatorname{Osc}%
\left( G,K\right) \rightarrow GL\left( \mathbb{R\times }\mathfrak{g}\right) $%
,%
\begin{equation*}
F\left( k,z,v\right) =\left( 
\begin{array}{ccc}
1 & 0 & 0 \\ 
z & \left. k\right\vert _{\mathfrak{p}} & 0 \\ 
v & \frac{1}{2}\operatorname{ad}_{z}\circ \left. k\right\vert _{\mathfrak{p}%
} & \left. k\right\vert _{\mathfrak{k}}%
\end{array}%
\right)
\end{equation*}%
is injective and preserves the corresponding multiplications. For the proof
of the assertion about $\operatorname{osc}\left( \mathfrak{g},\mathfrak{k}\right) $,
the relevant computations are as follows. First, that%
\begin{equation*}
\left. \frac{d}{ds}\right\vert _{0}\left( 
\begin{array}{ccc}
1 & 0 & 0 \\ 
sz & \left. \operatorname{Ad}\left( e^{su}\right) \right\vert _{\mathfrak{p}} & 0 \\ 
sv & \tfrac{1}{2}\operatorname{ad}_{sz}\circ \left. \operatorname{Ad}\left(
e^{su}\right) \right\vert _{\mathfrak{p}} & \left. \operatorname{Ad}\left(
e^{su}\right) \right\vert _{\mathfrak{k}}%
\end{array}%
\right) =\left( 
\begin{array}{ccc}
0 & 0 & 0 \\ 
z & \left. \operatorname{ad}_{u}\right\vert _{\mathfrak{p}} & 0 \\ 
v & \frac{1}{2}\left. \operatorname{ad}_{z}\right\vert _{\mathfrak{p}} & 
\left. \operatorname{ad}_{u}\right\vert _{\mathfrak{k}}%
\end{array}%
\right) \text{,}
\end{equation*}%
holds for $u,v\in \mathfrak{k}$ and $z\in \mathfrak{p}$, since $\operatorname{Ad}%
\left( e^{su}\right) =e^{s\operatorname{ad}_{u}}$. In addition, that if the 
right-hand side is denoted by $f\left( \operatorname{ad}_{u},z,v\right) $, 
then $f:\operatorname{osc}\left( 
\mathfrak{g},\mathfrak{k}\right) \rightarrow gl\left( \mathbb{R}\times \mathfrak{g}%
\right) $ preserves the Lie brackets, which is not difficult to check.
\end{proof} 

The group $\operatorname{Osc}\left( G,K\right) $ admits a canonical bi-invariant
pseudo-Riemannian metric $g$ of signature $\left( \dim \mathfrak{k},\dim 
\mathfrak{p}+\dim \mathfrak{k}\right) $, given at the identity by%
\begin{equation}
\left( -\kappa \right) g\left( \left( \operatorname{ad}_{u},z,v\right) ,\left( \operatorname{%
ad}_{\mu },\zeta ,\nu  \right) \right) =B\left( z,\zeta \right) +B\left(
u,\nu  \right) +B\left( v,\mu \right) \text{.}  \label{bi}
\end{equation}%
The bilinear form is nondegenerate since $B\left( \mathfrak{p,k}\right) =0$
and $\left. B\right\vert _{\mathfrak{p}\times \mathfrak{p}}$ and $\left.
B\right\vert _{\mathfrak{k}\times \mathfrak{k}}$ are nondegenerate. This is
part of the elementary theory of double extensions, but can be easily
verified by showing, using (\ref{bracketHD}), that $\operatorname{ad}_{W}$ 
is skew-symmetric with respect to $%
g $ for any $W\in \operatorname{osc}\left( \mathfrak{g},\mathfrak{k}\right) $, which
follows from the same property for $\mathfrak{g}$ with respect to $B$. 

\medskip

Next we give two examples that show that the groups introduced in this
section generalize $\operatorname{Osc}\left( \kappa \right) $. They are associated
to symmetric pairs that give rise to the same symmetric spaces, namely the
two-sphere and the hyperbolic plane; the first one is an effective
presentation; the second, not.

\begin{example} \label{Example1} \upshape
	For $\kappa =\pm 1$, let $\beta _{\kappa }$ be the quadratic
form on $\mathbb{R}^{3}$ given by $\beta _{\kappa }\left( x,y\right) =\kappa
x^{2}+y^{2}+z^{2}$ and let $G_{k}$ be the identity component of the
automorphism group of $\beta _{\kappa }$, that is, $G_{1}=SO\left( 3\right) $
and $G_{-1}=O_{o}\left( 1,2\right) $. Let $K=\left\{ 1\right\} \times
SO\left( 2\right) \equiv SO\left( 2\right) $. We have that $\left( G_{\kappa
},K\right) $ is a semisimple symmetric pair and $K$ is compact (in fact, $%
G_{\kappa }/K$ is a symmetric presentation of the sphere $S^{2}$ and the
hyperbolic plane $H^{2}$, for $\kappa =1$ and $\kappa =-1$, respectively).
Let $\mathfrak{g}_{\kappa }$ and $\mathfrak{k}$ be the Lie algebras of $%
G_{\kappa }$ and $K$, respectively, and let $\mathfrak{g}_{\kappa }=%
\mathfrak{p}_{\kappa }+\mathfrak{k}$ be the corresponding Cartan
decomposition. In order to describe them we introduce some notation: For $%
u\in \mathbb{R}$ and a column vector $z\in \mathbb{R}^{2}$ let 
\begin{equation*}
r\left( u\right) =\left( 
\begin{array}{cc}
\cos u & -\sin u \\ 
\sin u & \cos u%
\end{array}%
\right) \text{,\ \ \ }m_{\kappa }\left( z\right) =\left( 
\begin{array}{cc}
0 & -\kappa z^{T} \\ 
z & 0_{2}%
\end{array}%
\right) \ \ \ \text{and }\ j=\left( 
\begin{array}{cc}
0 & -1 \\ 
1 & 0%
\end{array}%
\right) \text{,}
\end{equation*}%
where $0_{2}$ is the $2\times 2$ zero matrix and $T$ denotes transpose. Then%
\begin{equation*}
K=\left\{ R\left( u\right) \mid u\in \mathbb{R}\right\} \text{,\ \ \ \ }%
\mathfrak{p}_{\kappa }=\left\{ m_{\kappa }\left( z\right) \mid z\in \mathbb{R%
}^{2}\right\} \text{ \ \ \ and \ \ \ \ }\mathfrak{k}=\left\{ J\left(
v\right) \mid v\in \mathbb{R}\right\} \text{,}
\end{equation*}%
where $R\left( u\right) =\operatorname{diag}\left( 1,r\left( u\right) \right) $ and $%
J\left( v\right) =\operatorname{diag}\left( 0,vj\right) $.
\end{example}

\begin{proposition}
The mapping $F_{\kappa }:\operatorname{Osc}\left( G_{\kappa },K\right) \rightarrow 
\operatorname{Osc}\left( \kappa \right) $,%
\begin{equation}
F_{\kappa }\left( \operatorname{Ad}\left( R\left( u\right) \right) ,m_{\kappa
}\left( z\right) ,J\left( v\right) \right) =\Phi _{\kappa }\left(
u,z,v\right)  \label{F1k}
\end{equation}%
is a well-defined Lie group isomorphism, as well as an isometry.

In particular, the mapping $\operatorname{Hei}\left( \kappa \right) \rightarrow
N\left( \mathfrak{g}_{\kappa },\mathfrak{k} \right) $, $\left(
z,v\right) \mapsto \left( m_{\kappa }\left( z\right) ,J\left( v\right)
\right) $ is an isometry of sub-Riemannian manifolds. \label{13}
\end{proposition}

\begin{proof} It can be deduced from (\ref{multiplicationOGK}), using that 
$\left. \operatorname{Ad}\right\vert _{K}$ is injective and the following identities:
\begin{eqnarray}
\operatorname{Ad}\left( R\left( u\right) \right) m_{\kappa }\left( z\right)
&=&m_{\kappa }\left( r\left( u\right) z\right) \text{,}  \label{f1} \\
\operatorname{Ad}\left( R\left( u\right) \right) \left( J\left( v\right) \right)
&=&J\left( v\right) \text{,}  \label{f2} \\
\left[ m_{\kappa }\left( z\right) ,m_{\kappa }\left( \zeta \right) \right]
&=&J\left( \kappa \det \left( z,\zeta \right) \right) \text{,}  \label{f3}
\end{eqnarray}%
which hold for all $u,v\in \mathbb{R}$, $z,\zeta \in \mathbb{R}^{2}$ and are
not difficult to check.

Next we deal with the second assertion. We call $\left\Vert X\right\Vert
=\left\langle X,X\right\rangle $, as usual in pseudo-Riemannian settings. We
have that $\left( dF_{\kappa }\right) _{\left( I_{\mathfrak{g}},0,0\right)
}\left( X\right) =\phi _{\kappa }\left( u,z,v\right) $, where $X=\left( 
\operatorname{ad}_{J\left( u\right) },m_{\kappa }\left( z\right) ,J\left( v\right)
\right) $.

For $\mathfrak{g}_{\kappa }$, the Killing form applied to two matrices $A$
and $B$ is given by $2B\left( A,B\right) =\operatorname{tr}\left( AB\right) $. 
Now, by (\ref{bi}) and (\ref{biOsc}), 
\begin{eqnarray*}
\left\Vert X\right\Vert &=&-\kappa B\left( m_{\kappa }\left( z\right)
,m_{\kappa }\left( z\right) \right) -\kappa 2B\left( J\left( u\right)
,J\left( v\right) \right) \\
&=&-\kappa \tfrac{1}{2}\operatorname{tr}\left( m_{\kappa }\left( z\right)
^{2}\right) -\kappa \operatorname{tr}\left( J\left( u\right) J\left( v\right) \right)
\\
&=&\left\vert z\right\vert ^{2}-\left( -2\kappa uv\right) =\left\vert
z\right\vert ^{2}+2\kappa uv \\
&=&\left\Vert \phi _{\kappa }\left( u,z^{T},v\right) \right\Vert \text{,}
\end{eqnarray*}%
as desired. For the last assertion, see Section \ref{SubSec1}.
\end{proof} 

\begin{example} \label{Example2} \upshape
For $\kappa =\pm 1$, let $h_{\kappa }$ be the Hermitian form
on $\mathbb{C}^{2}$ given by $h_{\kappa }\left( z,w\right) =\kappa
\left\vert z\right\vert ^{2}+\left\vert w\right\vert ^{2}$. Let $G_{\kappa }$
be the group of automorphisms of $h_{\kappa }$ with determinant one, that is 
$G_{1}=SU\left( 2\right) $ and $G_{-1}=SU\left( 1,1\right) $. For $z\in 
\mathbb{C}$ we set%
\begin{equation*}
D\left( z\right) =\operatorname{diag}\left( \bar{z},z\right) \text{ \ \ \ \ \ \ and
\ \ \ \ \ }p_{\kappa }\left( z\right) =\left( 
\begin{array}{cc}
0 & -\kappa \bar{z} \\ 
z & 0%
\end{array}%
\right) \text{,}
\end{equation*}%
both in $\mathbb{C}^{2\times 2}$. Let $K=\left\{ D\left( u\right) \mid u\in
S^{1}\subset \mathbb{C}\right\} \equiv U\left( 1\right) \equiv S^{1}$. Let $%
\mathfrak{g}_{\kappa }$ and $\mathfrak{k}$ be the Lie algebras of $G_{\kappa
}$ and $K$, respectively. Then $\mathfrak{g}_{\kappa }=\mathfrak{p}_{\kappa
}+\mathfrak{k}$ is the corresponding Cartan decomposition, where%
\begin{equation*}
\mathfrak{p}=\left\{ p_{\kappa }\left( z\right) \mid z\in \mathbb{C}\right\} 
\text{ \ \ \ and \ \ \ }\mathfrak{k}=\left\{ D\left( vi\right) \mid v\in 
\mathbb{R}\right\} \text{.}
\end{equation*}
\end{example}

\begin{proposition}
The mapping $F_{\kappa }:\operatorname{Osc}\left( G_{\kappa },U\left( 1\right)
\right) \rightarrow \operatorname{Osc}\left( \kappa \right) $,%
\begin{equation*}
F_{\kappa }\left( \operatorname{Ad}\left( D\left( e^{it}\right) \right) ,p_{\kappa
}\left( z\right) ,D\left( vi\right) \right) =\Phi \left( 2t,z,\tfrac{1}{2}%
v\right)
\end{equation*}%
is a well-defined Lie group isomorphism and a homothety of ratio $\frac{1}{4}
$.
\end{proposition}

\begin{proof} We comment that the factors $2$ and $\frac{1}{2}$ come from
the fact the action of $G_{\kappa }$ on $S^{2}$ or $H^{2}$ is not
effective. The good definition and the injectivity of $F_\kappa$ follow from the
following arguments: It is easy to check that $\operatorname{Ad}\left( k\right) =%
\operatorname{Ad}\left( h\right) $ if and only if $h=\pm k$. This implies that $%
\operatorname{Ad}\left( D\left( e^{it}\right) \right) =\operatorname{Ad}\left( D\left(
e^{i\tau }\right) \right) $ if and only if $\tau -t\in \pi \mathbb{Z}$ if
and only if $\Phi \left( 2t,z,v\right) =\Phi \left( 2\tau ,z,v\right) $.

Now, for $u\in S^{1}$ and $z\in \mathbb{C}$, we compute%
\begin{equation*}
\left( 
\begin{array}{cc}
\bar{u} & 0 \\ 
0 & u%
\end{array}%
\right) \left( 
\begin{array}{cc}
0 & -\kappa \bar{z} \\ 
z & 0%
\end{array}%
\right) \left( 
\begin{array}{cc}
u & 0 \\ 
0 & \bar{u}%
\end{array}%
\right) =\left( 
\begin{array}{cc}
0 & -\kappa \bar{z}\bar{u}^{2} \\ 
u^{2}z & 0%
\end{array}%
\right) \text{.}
\end{equation*}%
Hence, $\operatorname{Ad}\left( D\left( u\right) \right) p\left( z\right) =p\left(
u^{2}z\right) $. Similarly, for $z,\omega \in \mathbb{C}$,%
\begin{equation}
\left[ p_{\kappa }\left( z\right) ,p_{\kappa }\left( \omega \right) \right]
=D\left( \kappa \left( \bar{z}\omega -\bar{\omega}z\right) \right) =D\left(
2\kappa \operatorname{Im}\left( \bar{z}\omega \right) i\right) \text{.}  \label{ppD}
\end{equation}%
Using this (the last equation for $\omega =\zeta e^{2it}$) and (\ref%
{multiplicationOGK}), one can conclude that $F_{\kappa }$ is an isomorphism
of Lie groups.

We have that $\left( dF_{\kappa }\right) _{\left( I_{\mathfrak{g}%
},0,0\right) }\left( X\right) =\phi _{\kappa }\left( 2u,z,\frac{1}{2}%
v\right) $, where $X=\left( \operatorname{ad}_{D\left( ui\right) },p_{\kappa }\left(
z\right) ,D\left( vi\right) \right) $. For $\mathfrak{g}_{\kappa }$, the
Killing form applied to two matrices $A$ and $B$ is given by $2B\left(
A,B\right) =4\operatorname{tr}\left( AB\right) $. Then, by (\ref{bi}) and (\ref{biOsc}), 
\begin{eqnarray*}
\left\Vert X\right\Vert &=&-\kappa B\left( p_{\kappa }\left( z\right)
,p_{\kappa }\left( z\right) \right) -\kappa 2B\left( D\left( ui\right)
,D\left( vi\right) \right) \\
&=&-2\kappa \operatorname{tr}\left( p_{\kappa }\left( z\right) p_{\kappa }\left( z\right)
\right) -4\kappa \operatorname{tr}\left( D\left( ui\right) D\left( vi\right) \right) \\
&=&4\left\vert z\right\vert ^{2}+8\kappa uv=4\left( \left\vert z\right\vert
^{2}-2\left( 2u\tfrac{1}{2}v\right) \right) \\
&=&4\left\Vert \phi _{\kappa }\left( 2u,z,\tfrac{1}{2}v\right) \right\Vert 
\text{.}
\end{eqnarray*}%
Thus, the last assertion is true.
\end{proof} 

\subsection{$\operatorname{Osc}\left( G,K\right) $ as finite coverings of Stiefel
bundles in the classical rank-one case\label{Stiefel}}

In this section we present a third example, more elaborate, which
generalizes the presentation of $\operatorname{Osc}\left( \kappa \right) $ as the
unit circle bundle $\mathcal{D}^{\kappa ,1}$ in (\ref{UnitCB}) and allows us
to visualize $\operatorname{Osc}\left( G,K\right) $ in particular relevant cases,
specifically, when $\left( G,K\right) $ is the standard symmetric pair of a
compact classical rank-one symmetric space.

Let $\mathbb{F}$ $=\mathbb{R}$, $\mathbb{C}$ or the quaternions $\mathbb{H}$
and let $\left( G_{\mathbb{F}},K_{\mathbb{F}}\right) $ be the symmetric pair
of the corresponding simply connected two-point homogeneous space, that is, $%
\left( SO_{1+m},SO_{m}\right) $, $\left( SU_{1+m},U_{m}\right) $, $\left(
Sp_{1+m},Sp_{1}\times Sp_{m}\right) $, for $\mathbb{F}=\mathbb{R}$, $\mathbb{%
C}$ and $\mathbb{H}$, respectively. Let $\mathfrak{g}_{\mathbb{F}}=\mathfrak{%
p}_{\mathbb{F}}\oplus \mathfrak{k}_{\mathbb{F}}$ be the Cartan decomposition
of the Lie algebra of $G_{\mathbb{F}}$, where%
\begin{eqnarray}
\mathfrak{k}_{\mathbb{F}} &=&\left\{ \operatorname{diag}\left( a,A\right) \mid a\in 
\operatorname{Im}\mathbb{F}\text{, }A\in \mathbb{F}^{m\times m}\text{, }\bar{A}%
^{T}=-A\right\} \text{,}  \notag \\
\mathfrak{p}_{\mathbb{F}} &=&\left\{ p\left( z\right) =\left( 
\begin{array}{cc}
0 & -\bar{z}^{T} \\ 
z & 0_{m}%
\end{array}%
\right) \mid z\in \mathbb{F}^{m}\right\} \text{,}  \label{pz}
\end{eqnarray}%
for $\mathbb{F}=\mathbb{R}$ and $\mathbb{H}$. The extra condition $a=-\operatorname{%
tr}\left( A\right) $ must be added to the definition of $\mathfrak{k}_{%
\mathbb{C}}$. We call $N_{\mathbb{F}}=N\left( \mathfrak{g}_{\mathbb{F}},%
\mathfrak{k}_{\mathbb{F}}\right) $.

Let $\mathcal{D}$ be the distribution (\ref{graph}) on $N_{\mathbb{F}}$.
Notice that there is a canonical identification of $\mathfrak{p}_{\mathbb{F}%
} $ with $\mathbb{F}^{m}$. For each $p\in N_{\mathbb{F}}$, consider on $%
\mathcal{D}_{p}\subset \mathfrak{p}_{\mathbb{F}}\times \mathfrak{k}_{\mathbb{%
F}}$ the structure of vector space over $\mathbb{F}$ (on the right, as
usual, for $\mathbb{F}=\mathbb{H}$), the Hermitian inner product and the
orientation induced from $\mathfrak{p}_{\mathbb{F}}$ by the projection $\pi
_{\left( z,v\right) }:\mathcal{D}_{\left( z,v\right) }\rightarrow \mathfrak{p%
}_{\mathbb{F}}\equiv \mathbb{F}^{m}$.

In order to define Stiefel bundles for the sub-Riemannian structures on $%
\left( n_{\mathbb{F}},\mathcal{D}\right) $, we introduce some notation. Let
us denote $M=\dim _{\mathbb{R}}\mathbb{F}$, that is, $M=1,2$ or $4$ for $%
\mathbb{F}$ $=\mathbb{R}$, $\mathbb{C}$ or $\mathbb{H}$, respectively. For a
vector space over $\mathbb{C}$ or $\mathbb{H}$ we always consider the
orientation induced by the scalar multiplication.

Let $V$ be an oriented vector space over $\mathbb{F}$ (for $\mathbb{H}$, as
a right module, as usual) with $\dim _{\mathbb{H}}V=m$. Let $g_{\mathbb{F}}$
be a positive definite $\mathbb{F}$-Hermitian inner product and let $g=\operatorname{%
Re}\left( g_{\mathbb{F}}\right) $ be the corresponding real inner product on 
$V$.

\begin{definition}
A positively oriented orthonormal ordered basis $\mathcal{B}$ of $V$ for the 
\emph{real} inner product $g$ is said to be real $\mathbb{F}$-Hermitian if
there exists a direct sum $V=L_{1}\oplus \cdots \oplus L_{m}$ of oriented $%
\mathbb{F}$-lines such that $\mathcal{B}$ is the juxtaposition of positively
oriented real orthonormal bases of $L_{s}$, for $s=1,\dots ,m$. In other
words, a positively oriented basis $\mathcal{B}=\left\{ W_{1},\dots
,W_{Mm}\right\} $ is real $\mathbb{F}$-Hermitian if $\left\{ W_{Ms+1},\dots
,W_{Ms+M}\right\} $ is a positively oriented orthonormal basis of an $%
\mathbb{F}$-line $L_{s}$ for $s=0,\dots ,m-1$.

Let $\operatorname{St}\left( n_{\mathbb{F}},\mathcal{D}\right) $ be the set of all
real $\mathbb{F}$-Hermitian bases of $\mathcal{D}_{p}$ with $p\in \mathfrak{n%
}_{\mathbb{F}}$.
\end{definition}


The preceding definition is an effort for the sake of uniformity, but note
that a real $\mathbb{R}$-Hermitian basis is just a positively oriented
 orthonormal basis. Also, a real $\mathbb{C}$-Hermitian basis can be
thought of as a usual orthonormal basis for the $\mathbb{C}$-Hermitian inner
product, since given a unit vector $U$ in a complex line $L$, there exists a
unique $W$ such that $\left\{ U,W\right\} $ is a positively oriented
orthonormal basis of $L$. Besides, the requirement for the basis to be
positively oriented is void for $\mathbb{F}=\mathbb{C}$ and $\mathbb{H}$.

\begin{proposition}
\label{PropStiefel}The group $\operatorname{Osc}\left( SO_{1+m},SO_{m}\right) $ acts
simply transitively on $\operatorname{St}\left( n_{\mathbb{R}},\mathcal{D}\right) $.

The group $\operatorname{Osc}\left( SU_{1+m},U_{m}\right) $ acts transitively on $%
\operatorname{St}\left( n_{\mathbb{C}},\mathcal{D}\right) $ with isotropy subgroup
isomorphic to $\mathbb{Z}_{m+1}$.

The group $\operatorname{Osc}\left( Sp_{1+m},Sp_{1}\times Sp_{m}\right) $ acts
transitively on $\operatorname{St}\left( n_{\mathbb{H}},\mathcal{D}\right) $ with
isotropy subgroup isomorphic to $\mathbb{Z}_{2}$.
\end{proposition}

\begin{proof} First we study the action of the subgroup $\operatorname{Ad}\left(
K_{\mathbb{F}}\right) \times \left\{ \left( 0,0\right) \right\} $ of $\operatorname{%
Osc}\left( G_{\mathbb{F}},K_{\mathbb{F}}\right) $. Notice that $\mathcal{D}%
_{\left( 0,0\right) }=\mathfrak{p}_{\mathbb{F}}\times \left\{ 0\right\} $ is
preserved. For $\operatorname{diag}\left( a,A\right) \in K_{\mathbb{F}}$ we compute 
\begin{equation*}
\operatorname{Ad}\left( \operatorname{diag}\left( a,A\right) \right) p\left( z\right)
=p\left( Aza^{-1}\right) \text{.}
\end{equation*}%
Hence, under the identification $\mathbb{F}^{m}\equiv \mathfrak{p}_{\mathbb{F%
}}$, $z\mapsto p\left( z\right) $, the action of $K_{\mathbb{F}}$ reads $%
\operatorname{diag}\left( a,A\right) \cdot z=Aza^{-1}$. We consider different cases.

\textsl{Real case.} We have $A\in SO_{m}$ and $a=1$. Hence the restriction
of $\operatorname{Ad}_{G_{\mathbb{R}}}$ to $K_{\mathbb{R}}$ is injective. Also,
clearly $\operatorname{Ad}\left( K_{\mathbb{R}}\right) \equiv K_{\mathbb{R}}$ acts
simply transitively on the set of positively oriented orthonormal bases of $%
\mathfrak{p}_{\mathbb{R}}\equiv \mathbb{R}^{m}$.

\textsl{Complex case.} We have $A\in U_{m}$ and $a=1/\det_{\mathbb{C}}\left(
A\right) $. The action of $\operatorname{Ad}\left( K_{\mathbb{C}}\right) $ on $%
\mathfrak{p}_{\mathbb{C}}$ translates into the epimorphism of Lie groups $%
\psi :U_{m}\rightarrow U_{m}$, $\psi \left( A\right) =\det \left( A\right) A$%
, with kernel $\left\{ wI_{m}\mid w^{m+1}=1\right\} \equiv \mathbb{Z}_{m+1}$%
. Also, $U_{m}$ is clearly in bijection with the set of orthonormal bases
for the canonical Hermitian structure on $\mathbb{C}^{m}$.

\textsl{Quaternionic case.} We have that $A\in Sp_{n}$ and $a$ is a unit
quaternion. The action of $\operatorname{Ad}\left( K_{\mathbb{H}}\right) $ on $%
\mathfrak{p}_{\mathbb{H}}$ translates into $\operatorname{diag}\left( a,A\right)
\cdot z=Az\bar{a}$ for $z\in \mathbb{H}^{m}$, which is transitive on the set
of $\mathbb{H}$-lines. On each $\mathbb{H}$-line $L$, the mapping $%
S^{3}\times S^{3}\times L\rightarrow L$, $\left( p,q,z\right) \mapsto pz\bar{%
q}$ is the well-known double covering $S^{3}\times S^{3}\rightarrow SO\left(
L\right) \equiv SO_{4}$. So the action is transitive on the set of real $%
\mathbb{H}$-Hermitian bases of $\mathbb{H}^{m}$. Let us find the isotropy
subgroup at the ordered basis $\left\{
e_{1},e_{1}i,e_{1}j,e_{1}k,e_{2,}e_{2}i,\dots ,e_{m}k\right\} $. Suppose
that $Ae_{s}\ell \bar{a}=e_{s}\ell $ (or equivalently, $Ae_{s}=e_{s}\ell a%
\bar{\ell}$) for all $s=1,\cdots ,m$ and all $\ell =1,i,j,k$. This holds in
particular for $\ell =1$, hence $\ell a\bar{\ell}=a$ for all $\ell =1,i,j,k$%
. It is not difficult to show that this implies that $a=1$ or $a=-1$. This,
in turn, forces $A$ to be $I_{m}$ or $-I_{m}$, respectively. Thus, the
isotropy subgroup is isomorphic to $\mathbb{Z}_{2}$.

Now we deal with the action of the whole group $\operatorname{Osc}\left( G,K\right) $%
. For each $p\in n_{\mathbb{F}}$ we denote $\varpi _{p}=\left( \left. \pi
_{p}\right\vert _{\mathcal{D}_{p}}\right) ^{-1}:\mathfrak{p}_{\mathbb{F}%
}\rightarrow \mathcal{D}_{p}$. Notice that $\varpi _{\left( z,v\right)
}=\varpi _{\left( z,\nu  \right) }$, with the obvious identifications
coming from the invariance by vertical translations in $n_{\mathbb{F}}$. So,
by abuse of notation, we sometimes write $\varpi _{z}$ instead of $\varpi
_{\left( z,v\right) }$. An arbitrary element of $\operatorname{St}\left( G_{\mathbb{F%
}},K_{\mathbb{F}}\right) $ has the form%
\begin{equation*}
\left( \left( \zeta ,\nu  \right) ,\left( \varpi _{\zeta }\left(
W_{1}\right) ,\dots ,\varpi _{\zeta }\left( W_{Mm}\right) \right) \right) 
\text{,}
\end{equation*}%
with $\zeta \in \mathfrak{p}_{\mathbb{H}}$, $\nu  \in \mathfrak{k}_{%
\mathbb{H}}$ and $\left\{ W_{1},\dots ,W_{Mm}\right\} $ a real $\mathbb{F}$%
-Hermitian basis of $\mathfrak{p}_{\mathbb{F}}\equiv \mathbb{F}^{m}$. Now we
let $\left( k,z,v\right) \in \operatorname{Osc}\left( G_{\mathbb{F}},K_{\mathbb{F}%
}\right) $ act on it, obtaining%
\begin{equation*}
\left( \left( k,z,v\right) \cdot \left( \zeta ,\nu  \right) ,\left(
\varpi _{z+k\left( \zeta \right) }\left( k\left( W_{1}\right) \right) ,\dots
,\varpi _{z+k\left( \zeta \right) }\left( k\left( W_{Mm}\right) \right)
\right) \right) \text{,}
\end{equation*}%
where the $\mathfrak{n}_{\mathbb{F}}$-component is computed by (\ref%
{multiplicationOGK}), with $h$ the identity of $\operatorname{Ad}\left( K_{\mathbb{F}%
}\right) $ (notice that the $\mathfrak{p}_{\mathbb{F}}$-component of $\left(
k,z,v\right) \cdot \left( \zeta ,\nu  \right) $ is $z+k\left( \zeta
\right) $). Straightforward computations show that it is an action. For the
transitivity and the assertion on the isotropy subgroups, the relevant
arguments are those given at the beginning of the proof.
\end{proof} 

\subsection{Geodesics of $\operatorname{Osc}\left( G,K\right) $ and their
relationship with sub-Riemannian geodesics of $N\left( \mathfrak{g},\mathfrak{k}\right) 
$}

We obtain explicitly the geodesics of $\operatorname{Osc}\left( G,K\right) $ in
terms of the sub-Riemannian geodesics of $\left( N\left( \mathfrak{g},\mathfrak{k}%
\right) ,\mathcal{D}\right) $ with the metric $-\kappa B$ on $\mathfrak{p}%
\equiv \mathcal{D}_{\left( 0,0\right) }$ (see (\ref{graph})). The main
reason for including them in the article is that the former, that is, the one-parameter
subgroups of $\operatorname{Osc}\left( G,K\right) $, are involved in results in the
last section of the paper. Nevertheless, we believe that the computation is
interesting in its own right, for future applications.

\begin{theorem}
\label{ThmGamma}Let $z\in \mathfrak{p}$ and $u,v\in \mathfrak{k}$. The
geodesic $\gamma \left( s\right) =\exp \left( s\left( u,z,v\right) \right) $
of $\operatorname{Osc}\left( G,K\right) $ with $\gamma \left( 0\right) =e$ 
and initial velocity $\left( \operatorname{ad}_{u},z,v\right) $ is
given by%
\begin{equation}
\gamma _{\left( u,z,v\right) }\left( s\right) =\left( \operatorname{Ad}_{G}\left(
e^{su}\right) ,\sigma _{\left( z,u\right) }\left( s\right) +\left(
0,h_{u}\left( s\right) \left( v\right) \right) \right) \text{,}
\label{gammauxd}
\end{equation}%
where $\sigma _{\left( z,u\right) }$ is the sub-Riemannian geodesic $\sigma $
in $\left( N\left( \mathfrak{g},\mathfrak{k}\right) ,\mathcal{D}\right) $ with $\sigma
\left( 0\right) =0$ and initial covector $\left( z+u\right) ^{\flat }$
described in Theorem \emph{\ref{geoSubN}}.
\end{theorem}

\begin{corollary}
\label{Coro}The sub-Riemannian geodesics of $\left( N\left( \mathfrak{g},\mathfrak{k}%
\right) ,\mathcal{D}\right) $ can be expressed in terms of the exponential
map of the group $\operatorname{Osc}\left( G,K\right) $, containing $N\left( 
\mathfrak{g},\mathfrak{k}\right) $, as $\sigma _{\left( z,u\right) }=\mathcal{P}\circ
\gamma _{\left( u,z,0\right) }$, where $\mathcal{P}:\operatorname{Osc}\left(
G,K\right) \rightarrow N\left( \mathfrak{g},\mathfrak{k}\right) $ denotes the canonical
projection.
\end{corollary}

\smallskip

\noindent \textit{Proof of Theorem \ref{ThmGamma}. } 
Given $W$ in the Lie algebra of a Lie group, $s\mapsto \exp
\left( sW\right) $ is the unique smooth curve $E:\mathbb{R}$ $\rightarrow G$
with $E\left( 0\right) =e$ and $E^{\prime }\left( s\right) =dL_{E\left(
s\right) }\left( W\right) $. For simplicity we set $\gamma =\gamma _{\left(
u,x,v\right) }$. Clearly, $\gamma \left( 0\right) $ is the identity of $%
\operatorname{Osc}\left( G,K\right) $. It remains to verify the other condition.
Using (\ref{gammauxd}) and (\ref{multiplicationOGK}) we compute%
\begin{eqnarray*}
dL_{\gamma \left( s\right) }\left( \operatorname{ad}_{u},z,v\right) &=&\left. \tfrac{%
d}{dt}\right\vert _{0}\gamma \left( s\right) \left( e^{t\operatorname{ad}%
_{u}},tz,tv\right) \\
&=&\left. \tfrac{d}{dt}\right\vert _{0}\left( \operatorname{Ad}\left( e^{su}\right)
e^{t\operatorname{ad}_{u}},h_{u}\left( s\right) z+\operatorname{Ad}\left( e^{su}\right) tz,%
\mathcal{V}\left( s,t\right) \right) \text{,}
\end{eqnarray*}%
where $\mathcal{V}\left( s,t\right) =\int_{0}^{s}C\left( \tau \right) ~d\tau
+h_{u}\left( s\right) v+\operatorname{Ad}\left( e^{su}\right) \left( tv\right) +%
\frac{1}{2}tC\left( s\right) $ (the function $C$ was defined in the last paragraph of the 
proof of Theorem \ref{geoSubN}). Since 
\begin{equation*}
\left. \tfrac{d}{dt}\right\vert _{0}\mathcal{V}\left( s,t\right) =\operatorname{Ad}%
\left( e^{su}\right) \left( v\right) +\tfrac{1}{2}C\left( s\right) \text{,}
\end{equation*}%
we have that $dL_{\gamma \left( s\right) }\left( \operatorname{ad}_{u},z,v\right)
=\gamma ^{\prime }\left( s\right) $, as desired.
\hfill $\square$ 

\bigskip

For the sake of simplicity, by abuse of notation, we sometimes omit the
symbol $\operatorname{ad}$ in the $\operatorname{ad}\left( \mathfrak{k}\right) $-component
of $\operatorname{osc}\left( \mathfrak{g},\mathfrak{k}\right) $.

\begin{example} \upshape
\label{ExamplegammaGamma} We deduce from Theorem \ref{ThmGamma} the expression for the geodesics in $%
\operatorname{Osc}\left( \kappa \right) $ given in Proposition \ref{PropGeoHeiO},
via Example \ref{Example1}, whose notation we maintain. This could help to
conceptualize the corresponding formulas (\ref{sigmaPropo1}) and (\ref%
{gammaPropo1}) in coordinates of $\mathbb{C}\times \mathbb{R}$ and $\mathbb{R%
}\times \mathbb{C}\times \mathbb{R}$, respectively.
\end{example}

\begin{proposition}
\label{gammaGamma}The mapping $F_{\kappa }:\operatorname{Osc}\left( SO\left(
g_{\kappa }\right) ,SO\left( 2\right) \right) \rightarrow \operatorname{Osc}\left(
\kappa \right) $ in\emph{\ (\ref{F1k})} sends the geodesic through the
identity with initial velocity $\left( J\left( u\right) ,m_{\kappa }\left(
z\right) ,J\left( v\right) \right) $ to the geo\-desic with initial velocity $%
\phi _{\kappa }\left( u,z,v\right) $ obtained in \emph{(\ref{gammaPropo1})}.
\end{proposition}
\begin{proof} First, we compute $\sigma _{\left( m_{\kappa }\left(
z\right) ,J\left( u\right) \right) }$. By (\ref{f1}),%
\begin{eqnarray*}
h_{J\left( u\right) }\left( s\right) \left( m_{\kappa }\left( z\right)
\right) &=&\int_{0}^{s}\operatorname{Ad}\left( e^{tJ\left( u\right) }\right)
m_{\kappa }\left( z\right) ~dt\equiv \int_{0}^{s}m_{\kappa }\left(
e^{itu}z\right) ~dt \\
&=&m_{\kappa }\left( \int_{0}^{s}e^{itu}z~dt\right) =m_{\kappa }\left( 
\tfrac{1}{u}iz\left( 1-e^{isu}\right) \right)
\end{eqnarray*}%
if $u\neq 0$, and it equals $m_{\kappa }\left( sz\right) $ if $u=0$.

Next we compute the integrand in the second coordinate of (\ref{sigmaXU}) when $u\neq 0$: 
\begin{equation*}
\left[ h_{J\left( u\right) }\left( t\right) m_{\kappa }\left( z\right)
,h_{J\left( u\right) }^{\prime }\left( t\right) m_{\kappa }\left( z\right) %
\right] =\left[ m_{\kappa }\left( i\tfrac{z}{u}\left( 1-e^{itu}\right)
\right) ,m_{\kappa }\left( e^{itu}z\right) \right] \text{.}
\end{equation*}%
Using (\ref{f3}), we have that it equals%
\begin{eqnarray*}
J\left( \kappa \det \left( \tfrac{1}{u}\left( 1-e^{itu}\right)
iz,e^{itu}z\right) \right) &=&J\left( \tfrac{\kappa }{u}\left( \det \left(
iz,e^{itu}z\right) +\det \left( -e^{itu}iz,e^{itu}z\right) \right) \right) \\
&=&J\left( \tfrac{\kappa }{u}\left( \det \left( z,iz\right) -\det \left(
zi,e^{itu}z\right) \right) \right) \\
&=&J\left( \tfrac{\kappa }{u}\left( \left\vert z\right\vert ^{2}-\cos \left(
tu\right) \left\vert z\right\vert ^{2}\right) \right) \text{.}
\end{eqnarray*}%
Its integral between $0$ and $s$ is $J\left( \tfrac{\kappa }{u^{2}}%
\left\vert z\right\vert ^{2}\left( us-\sin \left( us\right) \right) \right) $%
. The case $u=0$ is similar.

Finally, $h_{J\left( u\right) }\left( s\right) \left( J\left( v\right)
\right) =\int_{0}^{s}\operatorname{Ad}\left( e^{tJ\left( u\right) }\right) J\left(
v\right) ~dt=J\left( vs\right) ,$ by (\ref{f2}). Consequently, $F_{\kappa }$
takes $\gamma _{\left( J\left( u\right) ,m_{\kappa }\left( z\right) ,J\left(
v\right) \right) }$ to the geodesic (\ref{gammaPropo1}).
\end{proof} 

\begin{example} \upshape
Actually, Theorem \ref{geoSubN} and Theorem \ref{ThmGamma} can be used to
obtain explicit expressions for sub-Riemannian geodesics of $N\left( 
\mathfrak{g},\mathfrak{k}\right) $ or geodesics in $\operatorname{Osc}\left(
G,K\right) $ not involving integrals or power series. We sketch the case $%
G=SU\left( 1+m\right) $, $K=U\left( m\right) $, presenting partial
computations that lead to the formulas, but omit to write them down 
for the sake of space, although they are elementary.

For $A\in \mathfrak{u}\left( m\right) $ we call $M\left( A\right) =\operatorname{diag%
}\left( -\operatorname{tr}\left( A\right) ,A\right) \in \mathfrak{k}$. We may
suppose without loss of generality that $u=M\left( Di\right) $, where $D\in 
\mathbb{R}^{m\times m}$ is a diagonal matrix (notice that any element of $%
\mathfrak{k}$ is conjugate by $K$ to $u$ of this form). For $z\in \mathbb{C}%
^{m}$, let $p\left( z\right) $ be as in (\ref{pz}). We compute $\operatorname{ad}%
_{u}\left( p\left( z\right) \right) =p\left( Eiz\right) $, where $E=D+%
\operatorname{tr}\left( D\right) I_{m}$ (here $I_{m}$ is the $\left( m\times
m\right) $-identity matrix). Hence,%
\begin{equation*}
\operatorname{Ad}\left( e^{tu}\right) p\left( z\right) =e^{t\operatorname{ad}_{u}}p\left(
z\right) =\sum\nolimits_{j=0}^{\infty }\tfrac{t^{j}}{j!}\left( \operatorname{ad}%
_{u}\right) ^{j}p\left( z\right) =p\left( e^{tEi}z\right) \text{.}
\end{equation*}%
Calling $e_{k}=E_{k,k}$, we have then that $h_{u}\left( s\right) p\left(
z\right) =p\left( y_{u}\left( s\right) \right) $, where
\begin{equation*}
 y_{u}\left(
s\right) _{k}=\int_{0}^{s}e^{te_{k}i}z_{k}~dt=\tfrac{1}{e_{k}}\left(
1-e^{se_{k}i}\right) z_{k}i
\end{equation*}
 if $e_{k}\neq 0$ and $y_{u}\left( s\right)
_{k}=sz_{k}$ if $e_{k}=0$. Other helpful auxiliary computations are $\left[
p\left( z\right) ,p\left( w\right) \right] =M\left(w\bar{z}^{T}-z\bar{w}^{T}\right) $ 
and $\operatorname{ad}_{u}\left( M\left( A\right) \right) =p\left( %
\left[ D,A\right] i\right) =\operatorname{diag}\left( 0,\left[ D,A\right] i\right) $.
\end{example}

\subsection{Sub-semi-Riemannian structures on $\operatorname{Osc}\left( G,K\right) $%
\label{Sub-semi}}

The following theorem generalizes Theorem \ref{subOsc} (a) and (d) from $\operatorname{Osc}%
\left( \kappa \right) $ to $\operatorname{Osc}\left( G,K\right) $ for symmetric
pairs $\left( G,K\right) $ with $G/K$ with arbitrary rank.

\begin{theorem}
\label{sub-semi}Consider on $\operatorname{Osc}\left( G,K\right) $ the canonical
bi-invariant pseudo-Riemannian metric $g$ and let $\mathcal{E}$ be the
left-invariant distribution on $\operatorname{Osc}\left( G,K\right) $ with $\mathcal{%
E}\left( e\right) =\mathfrak{h}^{\perp }$, where $\mathfrak{h}$ is either:

\smallskip

\emph{a)} the subspace%
\begin{equation}
\left\{ \left( \operatorname{ad}_{u},0,v\right) \mid u,v\in \mathfrak{t}\right\}
\oplus \left\{ \left( -\lambda \operatorname{ad}_{w},0,w\right) \mid w\in \mathfrak{z%
}\right\} \text{,}  \label{sumaDirecta}
\end{equation}%
with $\mathfrak{t}$ a maximal abelian subalgebra of $\left[ \mathfrak{k},%
\mathfrak{k}\right] $, $\mathfrak{z}$ the center of $\mathfrak{k}$ and $%
\lambda \neq 0$ a real number, or

\smallskip

\emph{b)} the subspace $\left\{ \left( 0,z,0\right) \in \operatorname{osc}\left( \mathfrak{g%
},\mathfrak{k}\right) \mid z\in \mathfrak{a}\right\} $, with $\mathfrak{a}$
a maximal abelian subalgebra of $\mathfrak{g}$ contained in $\mathfrak{p}$.

Let $g_{o}$ be the metric induced on $\mathcal{E}$ from $g$. Then $\left( 
\mathcal{E},g_{o}\right) $ is a sub-semi-Riemannian structure on $\operatorname{Osc%
}\left( G,K\right) $ and the normal geodesics through the identity are given
by $\gamma \left( s\right) =\exp \left( s\left( A+B\right) \right) \exp
\left( -sB\right) $, with $A\in \mathfrak{h}^{\perp }$ and $B\in \mathfrak{h}
$.
\end{theorem}

\begin{example} \label{lastexample} \upshape
We comment on the distributions in the theorem above for the oscillator
groups associated with compact classical two-point homogeneous spaces, as in
Section \ref{Stiefel} (in particular, their rank, equal to the dimension of $%
\mathfrak{a}$, is one). We use the same notation as in that section. 
We call $\mathfrak{k}^{\prime }=\left[ \mathfrak{k,k}\right] $, which is
a semisimple Lie algebra.

\smallskip

- We begin with the real case. If $m=\dim \mathfrak{p}=2$, then $\mathfrak{k}%
=\mathfrak{z}$ and $\mathfrak{k}^{\prime }=\mathfrak{t}=\left\{ 0\right\} $.
Then, under the isomorphism $F_{1}:\operatorname{Osc}\left( SO_{3},SO_{2}\right)
\rightarrow \operatorname{Osc}\left( 1\right) $ of (\ref{F1k}), the sub-Lorentzian
or sub-Riemannian structure $\mathcal{E}^{\lambda }$ on $\operatorname{Osc}\left(
1\right) $ in Theorem \ref{subOsc} corresponds to the one in (a) of the
theorem above. Besides, taking $\mathfrak{a}=\left\{ m\left( ae_{2}\right)
\mid a\in \mathbb{R}\right\} $, the sub-Lorentzian structure $\mathcal{E}%
^{o} $ in $\operatorname{Osc}\left( 1\right) $ corresponds to the one in (b) of the
preceding theorem.

Also in the real case, but with $m=\dim \mathfrak{p}>2$, we have $\mathfrak{k%
}^{\prime }=\mathfrak{k}$ and $\mathfrak{z}=\left\{ 0\right\} $. One can
take, for instance, $\mathfrak{t}=\left\{ \operatorname{diag}\left( 0,t_{1}j,\dots
,t_{k}j\right) \mid t_{1},\dots ,t_{k}\in \mathbb{R}\right\} $ if $m=2k$, $%
\mathfrak{t}$ $=\left\{ \operatorname{diag}\left( 0,t_{1}j,\dots ,t_{k}j,0\right)
\mid t_{1},\dots ,t_{k}\in \mathbb{R}\right\} $ if $m=2k+1$, and $\mathfrak{a%
}=\left\{ m\left( ae_{m}\right) \mid a\in \mathbb{R}\right\} $.

\smallskip

- In the complex case $\mathfrak{k}$ is not semisimple. We have $\mathfrak{k}%
^{\prime }=\left\{ \operatorname{diag}\left( 0,A\right) \mid A\in \mathfrak{su}\left( m\right)
\right\} \equiv \mathfrak{su}\left( m\right) $ and $\mathfrak{z}=\left\{ \operatorname{diag}%
\left( -mti,tI_{m}i\right) \mid t\in \mathbb{R}\right\} $. One can take 
\begin{equation*}
\mathfrak{t}=\left\{ \operatorname{diag}\left( 0,t_{1}i,\dots ,t_{m}i\right) \mid
t_{1},\dots ,t_{m}\in \mathbb{R}\text{, }t_{1}+\cdots +t_{m}=0\right\}
\equiv \mathbb{R}^{m-1}
\end{equation*}%
and $\mathfrak{a}=\left\{ p\left( ae_{m}\right) \mid a\in \mathbb{R}\right\}
\equiv \mathbb{R}$.

\smallskip

- In the quaternionic case $\mathfrak{k}$ is semisimple (although not
simple), hence $\mathfrak{k}^{\prime }=\mathfrak{k}$ and $\mathfrak{z}%
=\left\{ 0\right\} $. One can take $\mathfrak{t}=\left\{ \operatorname{diag}\left(
t_{0}i,\dots ,t_{m}i\right) \mid t_{l}\in \mathbb{R}\right\} \equiv \mathbb{R%
}^{1+m}$ and $\mathfrak{a}=\left\{ p\left( ae_{m}\right) \mid a\in \mathbb{R}%
\right\} \equiv \mathbb{R}$.
\end{example}

Before the proof of Theorem \ref{sub-semi}, 
we recall some properties of semisimple symmetric pairs $%
\left( G,K\right) $, introduced at the beginning of this section:%
\begin{equation*}
\left[ \mathfrak{p},\mathfrak{p}\right] =\mathfrak{k}\text{,\ \ }B\left( 
\mathfrak{p},\mathfrak{k}\right) =0,\text{\ }\left. B\right\vert _{\mathfrak{%
k}\times \mathfrak{k}}<0\text{ \ \ and }\left. -\kappa B\right\vert _{%
\mathfrak{p}\times \mathfrak{p}}>0\text{.}
\end{equation*}%
Now we list some well-known facts about the Lie subalgebras of $\mathfrak{g}$
involved in the statement of Theorem \ref{sub-semi}.

\smallskip

I)  We have that $
\mathfrak{k=k}^{\prime }\oplus \mathfrak{z}$  and $B\left( \mathfrak{k}%
^{\prime },\mathfrak{z}\right) =0$. In particular, (\ref{sumaDirecta}) is in fact a direct sum.

\smallskip

II) We define $\mathfrak{m}$ by $\mathfrak{k}^{\prime }
=\mathfrak{m }\oplus\mathfrak{t}$ with $B\left( \mathfrak{m,}\mathfrak{t}\right) =0$. 
Then there exist $r_{1},\dots ,r_{l}$, elements of $\mathfrak{t}$ spanning it, and a
basis $\left\{ x_{1},\dots ,x_{l},y_{1},\dots ,y_{l}\right\} $ of $\mathfrak{%
m}$ such that $\left[ x_{s},y_{s}\right] =r_{s}$ for all $s=1,\dots ,l$.

\smallskip

III) For each $\alpha \in \mathfrak{a}^{\ast }$ let 
\begin{equation*}
\mathfrak{g}_{\alpha }=\left\{ X\in \mathfrak{g\mid }\operatorname{ad}_{A}\left(
X\right) =\alpha \left( A\right) X\text{ for all }A\in \mathfrak{a}\right\} 
\text{.}
\end{equation*}%
An element $\alpha \in \mathfrak{a}^{\ast }$ is called a restricted root if $%
\mathfrak{g}_{\alpha }\neq \left\{ 0\right\} $. Let $\Delta $ be the set of
nonzero restricted roots, which is a finite set. Then $\mathfrak{g}=\mathfrak{g}%
_{0}+\sum\nolimits_{\alpha \in \Delta }\mathfrak{g}_{\alpha }$ holds (a
direct sum). Also, $\mathfrak{a}=\mathfrak{g}_{0}\cap \mathfrak{p}$ and $%
B\left( \mathfrak{g}_{\alpha },\mathfrak{g}_{\beta }\right) =0$ for any pair $\alpha, \beta $
of different restricted roots. 
For each $\alpha \in \Delta $ define $H_{\alpha }\in \mathfrak{%
a}$ by $\alpha \left( x\right) =2B\left( x,H_{\alpha }\right) $ 
(recall that $B$ was defined as one half of the Killing form). 
Then $%
\left\{ H_{\alpha }\mid \alpha \in \Delta \right\} $ spans $\mathfrak{a}$.

If $\alpha \in \Delta $, then $-\alpha \in \Delta $ and $\mathfrak{g}%
_{-\alpha }=\vartheta \left( \mathfrak{g}_{\alpha }\right) $. Given
 $\alpha\in\Delta $, choose a nonzero vector $X_{a}\in 
\mathfrak{g}_{\alpha }$, normalized in such a way that $\left\langle
X_{\alpha },X_{\alpha }\right\rangle =1/B\left( H_{\alpha },H_{\alpha
}\right) $. Putting $Y_{\alpha }=-\vartheta \left( X_{\alpha}\right) $, one has $%
\left[ X_{\alpha },Y_{\alpha }\right] =H_{\alpha }$.

\bigskip

\noindent \textit{Proof of Theorem \ref{sub-semi}. } In each case we have to verify the hypotheses of
Proposition \ref{PropRM}. By (\ref{bracketHD}), for both (a) and (b), $%
\mathfrak{h}$ is a subalgebra.

\medskip

a) Since $\mathfrak{p}\oplus \mathfrak{m\oplus t\oplus z}$ is an orthogonal
sum with respect to $B$, we have by (\ref{bi}) that%
\begin{equation}
\mathfrak{h}^{\perp }=\left\{ \left( u+\lambda w,z,v+w\right) \mid u,v\in 
\mathfrak{m}\text{, }w\in \mathfrak{z}\text{, }z\in \mathfrak{p}\right\}
\label{hperp1}
\end{equation}%
and, calling as before $\left\Vert W\right\Vert =g\left( W,W\right) $, that%
\begin{equation*}
-\kappa \left\Vert \left( u+\lambda w,z,v+w\right) \right\Vert =B\left(
z,z\right) +2B\left( u,v\right) +2\lambda B\left( w,w\right) \text{.}
\end{equation*}%
Now, $\mathfrak{p}\oplus \mathfrak{m\oplus z}$ is an orthogonal sum, $%
\lambda \neq 0$, $\left. -\kappa B\right\vert _{\mathfrak{p}\times \mathfrak{%
p}}>0$ and $\left. B\right\vert _{\mathfrak{k}\times \mathfrak{k}}<0$.
Hence, $g$ in nondegenerate on $\mathfrak{h}^{\perp }\times \mathfrak{h}%
^{\perp }$.

Next we verify that the distribution is bracket generating. It suffices to
show that any element $\left( u-\lambda w,0,v+w\right) $ of $\mathfrak{h}$
as in (\ref{sumaDirecta}) can be expressed as a sum of elements of $\left[ 
\mathfrak{h}^{\perp },\mathfrak{h}^{\perp }\right] $ and $\mathfrak{h}%
^{\perp }$. We write%
\begin{equation*}
\left( u-\lambda w,0,v+w\right) =\left( u,0,0\right) +\left( 0,0,v\right)
+\left( \lambda \left( -w\right) ,0,-w\right) +\left( 0,0,2w\right) \text{.}
\end{equation*}%
Since $u\in \mathfrak{t}$, by (II) above, $u=\sum_{s=1}^{l}t_{s}r_{s}=%
\sum_{s=1}^{l}t_{s}\left[ x_{s},y_{s}\right] $ for some $t_{1},\dots
,t_{l}\in \mathbb{R}$ and $x_{s},y_{s}\in \mathfrak{m}$. Hence, by (\ref%
{bracketHD}), 
\begin{equation*}
\left( u,0,0\right) =\sum\nolimits_{s=1}^{l}t_{s}\left( \left[ x_{s},y_{s}%
\right] ,0,0\right) =\sum\nolimits_{s=1}^{l}t_{s}\left[ \left(
x_{s},0,0\right) ,\left( y_{s},0,0\right) \right] \text{,}
\end{equation*}%
which belongs to $\left[ \mathfrak{h}^{\perp },\mathfrak{h}^{\perp }\right] $
by (\ref{hperp1}). Similar arguments show that $\left( 0,0,v\right) \in %
\left[ \mathfrak{h}^{\perp },\mathfrak{h}^{\perp }\right] $ for $v\in 
\mathfrak{t}$.

Also, $\left( \lambda \left( -w\right) ,0,-w\right) \in \mathfrak{h}^{\perp
} $, since $w\in \mathfrak{z}$. Finally, $2w\in \mathfrak{z}\subset 
\mathfrak{k}$ $=\left[ \mathfrak{p,p}\right] $. Hence, by (\ref{bracketHD}), 
$\left( 0,0,2z\right) =\left[ \left( 0,\xi ,0\right) ,\left( 0,\eta
,0\right) \right] $ for some $\xi ,\eta \in \mathfrak{p}$ and so it belongs
to $\left[ \mathfrak{h}^{\perp },\mathfrak{h}^{\perp }\right] $ by (\ref%
{hperp1}). This completes the proof that $\mathfrak{h}^{\perp }$ is bracket
generating.

\medskip

b) We have that $g\left( \left( 0,x,0\right) ,\left( 0,y,0\right) \right)
=-\kappa B\left( x,y\right) $. So $\left. g\right\vert _{\mathfrak{h}^{\perp
}\times \mathfrak{h}^{\perp }}$ is
nondegenerate, since $\left. -\kappa B\right\vert _{\mathfrak{p}^{\perp
}\times \mathfrak{p}^{\perp }}$ is positive
definite. This implies that $\left. g\right\vert _{\mathfrak{h}^{\perp
}\times \mathfrak{h}^{\perp }}$ is nondegenerate, since $g$ is nondegenerate.

By (\ref{bi}), $\mathfrak{h}^{\perp }=\left\{ \left( u,\eta ,\zeta \right)
\mid u,\zeta \in \mathfrak{k}\text{, }\eta \in \mathfrak{l}\right\} $, where 
$\mathfrak{l}$ is the orthogonal complement of $\mathfrak{a}$ in $\mathfrak{p%
}$ with respect to $B$. Next we check that $\mathfrak{h}^{\perp }$ is
bracket generating.

Given $\left( 0,x,0\right) \in \mathfrak{h}$, that is, $x\in \mathfrak{a}$,
it suffices to find $u\in \mathfrak{k}$ and $y\in \mathfrak{l}$ such that $%
\left[ u,y\right] =x$. Indeed, in that case, by (\ref{bracketHD}), 
\begin{equation*}
\left[ \left( u,0,0\right) ,\left( 0,y,0\right) \right] =\left( 0,\left[ u,y%
\right] ,0\right) =\left( 0,x,0\right) \in \mathfrak{h}\text{,}
\end{equation*}%
with $\left( u,0,0\right) ,\left( 0,y,0\right) $ both in $\mathfrak{h}%
^{\perp }$.

We use the properties of the restricted roots recalled above, in (III).
Since the vectors $H_{\alpha }$ span $\mathfrak{a}$, we may suppose without
loss of generality that $x=-\frac{1}{2}H_{\alpha }$ for some $\alpha \in
\Delta $. We take $X_{a}$ as in (III) and write $X_{\alpha }=x_{\alpha
}+z_{\alpha }$, with $x_{\alpha }\in \mathfrak{p}$ and $z_{\alpha }\in 
\mathfrak{k}$. Then we have $Y_{\alpha }=-x_{\alpha }+z_{\alpha }$ and 
\begin{equation*}
H_{\alpha }=\left[ x_{\alpha }+z_{\alpha },-x_{\alpha }+z_{\alpha }\right] =%
\left[ x_{\alpha },z_{\alpha }\right] -\left[ z_{\alpha },x_{\alpha }\right]
=2\left[ x_{\alpha },z_{\alpha }\right] \text{.}
\end{equation*}

Let us see that $u=z_{\alpha }$ and $y=x_{\alpha }$ satisfy the required
conditions. In fact, $\left[ u,y\right] =\left[ z_{\alpha },x_{\alpha }%
\right] =-\frac{1}{2}H_{\alpha }=x$. It remains to show that $x_{\alpha }\in 
\mathfrak{l}$, that is, that $B\left( \mathfrak{a},x_{\alpha }\right) =0$.
Given $\xi \in \mathfrak{a}$, we have $B\left( \xi ,x_{\alpha }\right)
=B\left( \xi ,x_{\alpha }+z_{\alpha }\right) $ since $B\left( \mathfrak{p,k}%
\right) =0$. Now, on the one hand, $\xi \in \mathfrak{g}_{0}$, since $%
\mathfrak{a}=\mathfrak{g}_{0}\cap \mathfrak{p}$. On the other hand, $%
x_{\alpha }+z_{\alpha }\in \mathfrak{g}_{\alpha }$. Thus, $B\left( \xi
,x_{\alpha }\right) =0$, as desired, since we know that $B\left( \mathfrak{g}%
_{\alpha },\mathfrak{g}_{\beta }\right) =0$ if $\alpha \neq\beta  $ and in
particular, $B\left( \mathfrak{g}_{\alpha },\mathfrak{g}_{0}\right) =0$ if $%
\alpha \in \Delta $. \hfill $\square $

\begin{remark} \upshape
The space $\left( n_{\mathbb{R}},\mathcal{D}\right) $ presented in Section %
\ref{Stiefel} is a sub-Riemannian model space in the sense of Grong \cite%
{Grong}. In Section 5.2 of that paper, a canonical sub-Riemannian structure
on the bundle $F^{O}\left( \mathcal{D}\right) $ of orthonormal frames of a
model space $\left( M,\mathcal{D}\right) $ is introduced, provided the
homology of a certain canonical partial connection is not trivial. We
comment that this is not the case for $\operatorname{St}\left( n_{\mathbb{R}},%
\mathcal{D}\right) =F^{\text{O}}\left( \mathcal{D}\right) $, by Proposition
3.7 in \cite{Grong}. In particular, the structures on $\operatorname{Osc}\left(
SO_{1+m},SO_{m}\right) $ (which can be identified with $\operatorname{St}\left( n_{%
\mathbb{R}},\mathcal{D}\right) $ by Proposition \ref{PropStiefel}) defined
in Theorem \ref{sub-semi} are not related to \cite{Grong}.
\end{remark}

\section{Appendix 
	\label{Appendix}}

\subsection{Hamiltonian geodesics} Let $M$ be a smooth manifold with
a Hamiltonian $H$, that is, a smooth function $H:T^{\ast }M\rightarrow 
\mathbb{R}$. Given $p\in M$ and $\alpha _{o}\in T_{p}^{\ast }M$, there is distinguished
smooth curve $\sigma :I\rightarrow M$, defined on an open interval $I$,
associated with $\alpha _{o}$, called the Hamiltonian geodesic of $\left(
M,H\right) $ with initial covector $\alpha _{o}$. It is the curve $\pi \circ
\alpha :I\rightarrow M$, where $\pi :T^{\ast }M\rightarrow M$ is the
canonical projection and $\alpha :I\rightarrow T^{\ast }M$ is
the maximal integral curve of the symplectic gradient vector field sgrad$%
\left( H\right) $ of $H$, with $\alpha \left( 0\right) =\alpha _{o}$.

We recall the definition of the symplectic gradient: The cotangent bundle of 
$M$ has a canonical symplectic structure $\Omega =d\theta $, where $\theta $
is the 1-form on $T^{\ast }M$ given by $\theta \left( \eta \right) =\beta
\left( d\pi _{\beta }\left( \eta \right) \right) $, if $\eta \in T_{\beta
}\left( T^{\ast }M\right) $. Then, sgrad$\left( H\right) $ is the unique
vector field $\xi $ on $T^{\ast }M$ such that $\Omega _{\beta }\left( \eta
,\xi \right) =dH_{\beta }\left( \eta \right) $ for all $\beta \in T^{\ast }M$
and $\eta \in T_{\beta }\left( T^{\ast }M\right) $.

\subsection{Sub-Riemannian and sub-Lorentzian geometry}
Let $\mathcal{D}$ be a bracket generating distribution on a
manifold $M$, that is, for every $p\in M$ there exist vector fields defined
on an open neighborhood of $p$ such that evaluated at $p$ form a basis of $%
T_{p}M$ and that all their successive commutators evaluated at $p$ span $%
T_{p}M$. A sub-Riemannian structure on $\left( M,\mathcal{D}\right) $ is an
assignment $g$ of a positive definite inner product $g_{p}$ on each subspace 
$\mathcal{D}_{p}$, varying smoothly with $p\in M$.

The inner product $g_{p}$ may be equivalently presented as a liner map $%
b_{p}:\left( T_{p}M\right) ^{\ast }\rightarrow T_{p}M$ with $b^{\ast }=b$
and image $\mathcal{D}_{p}$, through the relationship 
\begin{equation*}
g_{p}\left( b_{p}\left( \alpha \right) ,b_{p}\left( \beta \right) \right)
=\alpha \left( b_{p}\left( \beta \right) \right) =\beta \left( b_{p}\left(
\alpha \right) \right)
\end{equation*}%
for all $\alpha ,\beta \in \left( T_{p}M\right) ^{\ast }$ (the hypotheses on 
$b_{p}$ ensure that $g_{p}$ is well defined). The Hamiltonian associated
with $b$ is then the function $H:\left( T^{\ast }M\right) \rightarrow 
\mathbb{R}$, $H\left( \alpha \right) =\frac{1}{2}\alpha \left( b\left(
\alpha \right) \right) $. Hamiltonian geodesics are called normal geodesics
as well.

If $M$ is a Lie group $G$ with Lie algebra $\mathfrak{g}$ 
and the sub-Riemannian structure $\left( G,%
\mathcal{D},g\right) $ is left-invariant, then $T^{\ast }G$ may be naturally
identified with $G\times \mathfrak{g}^{\ast }$
 and the Hamiltonian turns out to be invariant: $H:T^{\ast
}G\equiv G\times \mathfrak{g}^{\ast }\rightarrow \mathbb{R}$, $H\left(
p,\alpha \right) =\frac{1}{2}\alpha \left( b_{e}\left( \alpha \right)
\right) $. Under these conditions, given $\alpha
_{o}\in \mathfrak{g}^{\ast }$, the Hamiltonian geodesic $\sigma $ with
initial covector $\alpha_{o}$ is the unique smooth curve in $G$ such that the curve $%
\left( \sigma ,\alpha \right) $ in $G\times \mathfrak{g}^{\ast }\equiv
T^{\ast }G$ is the maximal solution of the equation system%
\begin{equation*}
\sigma ^{\prime }\left( s\right) =dL_{\sigma \left( s\right) }\left( b\left(
\alpha \left( s\right) \right) \right) \text{\ \ \ \ \ \ \ and \ \ \ \ \ }%
\alpha ^{\prime }\left( s\right) =\left( \text{$\operatorname{ad}$}_{b_{o}\left(
\alpha \left( s\right) \right) }\right) ^{\ast }\alpha \left( s\right) \text{%
,}
\end{equation*}%
with initial conditions $\sigma \left( 0\right) =e$ and $\alpha \left(
0\right) =\alpha _{o}$.

The length of an admissible curve $\sigma :\left[ a,b\right] \rightarrow N$
is defined by $\text{length}\left( \sigma \right) =\int_{a}^{b}\sqrt{g\left(
\sigma ^{\prime }\left( t\right) ,\sigma ^{\prime }\left( t\right) \right) }%
~dt$. For arbitrary $p,q\in M$, the expression 
\begin{equation*}
d\left( p,q\right) =\inf \left\{ \text{length}\left( \sigma \right) \mid
\sigma \text{ is a piecewise admissible curve joining }p\text{ with }%
q\right\}
\end{equation*}%
defines a distance on $M$. It is well-known that a Hamiltonian geodesic
 $\gamma $ of $M$ minimizes the length locally,
that is, if $c<d$ in the domain of $\gamma $ are close enough, then the
length of $\left. \gamma \right\vert _{\left[ c,d\right] }$ equals the
distance between $\gamma \left( c\right) $ and $\gamma \left( d\right) $.
Also, that there may exist constant speed locally length minimizing curves
in $M$ which are not normal geodesics.

The case when the inner product $g$ on $\mathcal{D}$ is no longer positive
definite, but still nondegenerate, is also interesting. As far as we know, the
sub-Lorentzian case (when the signature is $\left( 1,k\right) $ for some $%
k\in \mathbb{N}$) has been first addressed by M.\ Grouchiness in \cite{Grocho}%
. Proposition 4.1 there states that time-like Hamiltonian geodesics (an
admissible curve $\sigma $ is time-like if $g\left( \sigma ^{\prime },\sigma
^{\prime }\right) <0$) are locally length maximizing, a property shared with
time-like Lorentzian geodesics (the length of a time-like curve is defined
as in the sub-Riemannian case, but changing the sign of the radicand). Other
articles on sub-Lorentzian geometry are \cite{Chang, KorolkoSR,
GrongVasilevJGM, SachkovSachkova, LokuPodo}.

Metrics on bracket generating distributions with signature $\left( \ell
,k\right) $ with $1<\ell \leq k$ appear naturally, for example, when
studying the rolling motion of one pseudo-Riemannian manifold over another,
as in \cite{Leite}. Articles devoted to this structure, called
sub-semi-Riemannian or sub-pseudo-Riemannian, are, for instance, \cite{GM-K, GrochoSPsR}.

\smallskip

In this paper the explicit form of normal geodesics of $\operatorname{Osc}\left(
\kappa \right) $ and $\operatorname{Osc}\left( G,K\right) $ are obtained using one
the simplest results for this purpose, that has been exploited extensively,
and of which we provide new applications:

\begin{proposition}
\label{PropRM}\emph{\cite{RM,Brockett}} Let $G$ be a Lie group endowed with
a bi-invariant pseudo-Rieman\-nian metric $g$. Let $\mathfrak{h}$ be a
subalgebra of $\mathfrak{g}$, the Lie algebra of $G$. Suppose that $%
\mathfrak{d}=\mathfrak{h}^{\perp }$ is bracket generating and the metric on $%
\mathfrak{d}$ is nondegenerate. Then all the normal geodesics of the
sub-semi-Riemannian manifold $\left( G,\mathfrak{d},\left. g\right\vert _{%
\mathfrak{d}\times \mathfrak{d}}\right) $ through the identity have the form 
$t\mapsto \exp \left( t\left( A+B\right) \right) \exp \left( -tB\right) $,
with $A\in \mathfrak{d}$ and $B\in \mathfrak{h}$.
\end{proposition}

This is Proposition 11.19 in \cite{RM} or the main result of \cite{Brockett}%
. The hypotheses in those references are more restrictive, namely, that the
structure is sub-Riemannian and the Lie subgroup with Lie algebra $\mathfrak{%
k}$ is closed. However, one can see in \cite{Srdjan} (although it deals mainly with compact Lie algebras) 
that those requirements are not necessary. As
far as we know, this recent paper contains the most general results in the
spirit of Proposition \ref{PropRM}.

\bigskip

\noindent Mauricio Godoy Molina; mauricio.godoy@ufrontera.cl

\noindent Departamento de Matem\'atica y Estad\'{\i}stica, Universidad de la Frontera

\noindent Avenida Francisco Salazar 01145, Temuco, Chile

\medskip 

\noindent Marcos Salvai; marcos.salvai@unc.edu.ar

\noindent \textsc{famaf} (Universidad Nacional de C\'ordoba) and \textsc{ciem} (\textsc{conicet})

\noindent Ciudad Universitaria, (\textsc{x5000hua}) C\'{o}rdoba, Argentina


\begin{thebibliography}{99}
\bibitem{AgrachevLibro} Andrei Agrachev, Davide Barilari and Ugo Boscain. A
Comprehensive Introduction to Sub-Riemannian Geometry. Vol. 181. Cambridge
Studies in Advanced Mathematics. Cambridge University Press, Cambridge, 2020.

\bibitem{AgrachevAbnormal} A. Agrachev, Yu. Sachkov. Control theory from the
geometric viewpoint. Encyclopedia of Mathematical Sciences, 87. Control
Theory and Optimization, II. Sp ringer-Verlag, Berlin, 2004.

\bibitem{Biggs} Biggs, Rory; Remsing, Claudiu C., Some remarks on the
oscillator group. Differ. Geom. Appl. 35, Suppl., 199-209 (2014).

\bibitem{BiggsNagy} Biggs, Rory; Nagy, P\'{e}ter T., On extensions of
sub-Riemannian structures on Lie groups. Differ. Geom. Appl. 46, 25-38
(2016).

\bibitem{Brockett} Brockett, R.W.: Explicitly solvable control problems with
nonholonomic constraints. In Decision and control, 1999, Proc. 38th IEEE
Conf. on, vol. 1, pp. 13--16, IEEE, 1999.

\bibitem{SachkovSol} Butt, Yasir Awais; Sachkov, Yuri L.; Bhatti, Aamer
Iqbal, Cut locus and optimal synthesis in sub-Riemannian problem on the Lie
group SH(2). J. Dyn. Control Syst. 23, No. 1, 155-195 (2017).

\bibitem{Chang} Chang, Der-Chen; Markina, Irina; Vasil'ev, Alexander,
Sub-Lorentzian geometry on anti-de Sitter space. J. Math. Pures Appl. (9)
90, No. 1, 82-110 (2008).

\bibitem{LinoGrama} Da Silva, A., Grama, L., Novaes, D. D., Tusco, M. Q.
(2026). Maxwell Strata in the sub-Riemannian problem on solvable,
nonnilpotent regular three-dimensional Lie groups. arXiv preprint
arXiv:2602.17782.

\bibitem{NorbertDorr} D\"{o}rr, N., \textsl{A note on the oscillator group}.
Semin. Sophus Lie 2, No. 1, 31-38 (1992).

\bibitem{Favre} Favre, G., and L. J. Santharoubane, Symmetric, invariant,
nondegenerate bilinear form on a Lie algebra, J. Algebra 105 (1987),
451--464.

\bibitem{Claudio} Falbel, Elisha; Gorodski, Claudio, On contact
sub-Riemannian symmetric spaces. Ann. Sci. \'{E}c. Norm. Sup\'{e}r. (4) 28,
No. 5, 571-589 (1995).

\bibitem{FVV} Falbel, E.; Veloso, J. M.; Verderesi, J. A., Constant
curvature models in sub-Riemannian geometry. Mat. Contemp. 4, 119-125 (1993).

\bibitem{GM-K} Godoy Molina, Mauricio; Korolko, Anna; Markina, Irina.
Sub-semi-Riemannian geometry of general H-type groups. Bull. Sci. Math. 137,
No. 6, 805-833 (2013).

\bibitem{Grocho} M. Grochowski, Geodesics in the sub-Lorentzian geometry,
Bull. Polish Acad. Sci. 50 (2) (2002).

\bibitem{GrochoSPsR} Grochowski, Marek; Kry\'{n}ski, Wojciech, On contact
sub-pseudo-Riemannian isometries. ESAIM, Control Optim. Calc. Var. 23, No.
4, 1751-1765 (2017).

\bibitem{Grong} Grong, Erlend, Model spaces in sub-Riemannian geometry.
Commun. Anal. Geom. 29, No. 1, 77-113 (2021).

\bibitem{GrongVasilevJGM} Grong, Erlend; Vasil'ev, Alexander, Sub-Riemannian
and sub-Lorentzian geometry on SU(1,1) and on its universal cover. J. Geom.
Mech. 3, No. 2, 225-260 (2011).

\bibitem{Srdjan} Bo\v{z}idar Jovanovi\'c, Tijana \v{S}ukilovi\'c, Srjdan Vukmirovi\'c, 
Normal sub-Riemannian geodesics related to filtrations of Lie algebras,
The Journal of Geometric Analysis, Vol. 36, No. 6: 215, 2026.

\bibitem{Kath}  Ines Kath, Matti Lyko. Pseudo-Riemannian symmetric spaces of signature (2,2).
Transform. Groups 31 (2026), no. 2, 1701-1751

\bibitem{KorolkoSR} Korolko, Anna; Markina, Irina, Nonholonomic Lorentzian
geometry on some H-type groups. J. Geom. Anal. 19, No. 4, 864-889 (2009).

\bibitem{LibroLeDonne} Le Donne, Enrico, Metric Lie groups. Carnot-Carath%
\'{e}odory spaces from the homogeneous viewpoint. Graduate Texts in
Mathematics 306. Cham: Springer. xvi, 480 p. (2025).

\bibitem{LokuNnS} Ladeishchikov, Evgeny A.; Lokutsievskiy, Lev V.; Prilepin, Nikita V.
Explicit formulae for extremals in sub-Lorentzian and Finsler problems on 2D
and 3D Lie groups. 
Sb. Math. 216, No. 12, 1713-1753 (2025); translation from Mat. Sb. 216, No.
12, 79-124 (2025).

\bibitem{LokuPodo} Lokutsievskiy, L. V.; Podobryaev, A. V., Existence
theorem for sub-Lorentzian problems. 
J. Dyn. Control Syst. 30, No. 2, 1-12 (2024).

\bibitem{Leite} Markina, Irina; Leite, F\'{a}tima Silva. Introduction to the
intrinsic rolling with indefinite metric. Commun. Anal. Geom. 24, No. 5,
1085-1106 (2016).

\bibitem{Medina} A. Medina and P. Revoy. Alg\`{e}bres de Lie et produit
scalaire invariant. Ann. Sci. Ecole Norm. Sup. (4), 18:553--561, 1985.

\bibitem{Myani} Myasnichenko, Oleg, Nilpotent $\left( n,n\left( n+1\right)
/2\right) $ sub-Riemannian problem.J. Dyn. Control Syst. 12, No. 1, 87-95
(2006).

\bibitem{Monroy} Monroy-P\'{e}rez, F.; Anzaldo-Meneses, A., The step-2
nilpotent $\left( n,n\left( n+1\right) /2\right) $, sub-Riemannian geometry.
J. Dyn. Control Syst. 12, No. 2, 185-216 (2006).

\bibitem{AnnamriaCruz} Montanari, Annamaria; Morbidelli, Daniele, On the
subRiemannian cut locus in a model of free two-step Carnot group. Calc. Var.
Partial Differ. Equ. 56, No. 2, Paper No. 36, 26 p. (2017).

\bibitem{Gabriela} P. R. Montenegro and G. P. Ovando, 
Closed Geodesics and Isometries on Compact Lorentz Solvmanifolds.
Mathematische Nachrichten 299, no. 7 (2026): 1716–1736.

\bibitem{RM} Richard Montgomery. A Tour of Subriemannian Geometries,
Their Geodesics and Applications. Vol. 91. Mathematical Surveys and
Monographs. American Mathematical Society, Aug. 2002, pp. xx+259. DOI:
10.1090/surv/091.


\bibitem{OvandoTorino} G. P. Ovando. Lie algebras with ad-invariant metrics.
A survey - guide. Rend. Semin. Mat., Univ. Politec. Torino, 74(1):243--268,
2016.

\bibitem{RizziCut} Rizzi, Luca; Serres, Ulysse, On the cut locus of free,
step two Carnot groups. Proc. Am. Math. Soc. 145, No. 12, 5341-5357 (2017).

\bibitem{SachkovSurveyLie} Sachkov, Yu. L., Control theory on Lie groups. J.
Math. Sci., New York 156, No. 3, 381-439 (2009); translation from Sovrem.
Mat. Fundam. Napravl. 27, 5-59 (2007).

\bibitem{SachkovSE2} Sachkov, Yuri L. Cut locus and optimal synthesis in the
sub-Riemannian problem on the group of motions of a plane. ESAIM, Control
Optim. Calc. Var. 17, No. 2, 293-321 (2011).

\bibitem{SachkovSachkova} Sachkov, Yuri L.; Sachkova, Elena F.
Sub-Lorentzian distance and spheres on the Heisenberg group. J. Dyn. Control
Syst. 29, No. 3, 1129-1159 (2023).

\bibitem{Streater} Streater, R.F. The representations of the
oscillator group. Commun. Math. Phys. 4, 217-236 (1967).
\end{thebibliography}
\end{document}